\documentclass[10pt]{article}
	\usepackage{amsmath,amssymb,amsthm,mathrsfs}
	\usepackage{graphicx}
	\usepackage[papersize={210mm,297mm},top=2cm, bottom=2cm, left=2cm , right=2cm]{geometry}
	\usepackage[utf8]{inputenc}
	\usepackage[T1]{fontenc}
	\usepackage{bmpsize}
	\usepackage{fancybox}
	\usepackage{array}
	\usepackage{geometry}
	\usepackage{enumitem}
	\usepackage{tikz}
	\usepackage[pagewise]{lineno}
	\usepackage{graphicx}
	\theoremstyle{remark}
	\newtheorem{remark}{Remark}

	\theoremstyle{definition}
	\newtheorem{definition}{D{e}finition}[section]
	\theoremstyle{plain}
	\newtheorem{proposition}[definition]{Proposition}
	
	\newtheorem{theorem}[definition]{Th{e}or{e}m}
	\newtheorem{corollary}[definition]{Corollary}
	\newtheorem{lemma}[definition]{Lemma}
	
\begin{document}
	
	\title{Uniqueness results in the inverse spectral Steklov problem}
	\author{Germain Gendron \\[12pt]
		  \small  Laboratoire de Math\'ematiques Jean Leray, UMR CNRS 6629, \\ \small 2 Rue de la Houssini\`ere BP 92208, F-44322 Nantes Cedex 03. \\
		\small Email: germain.gendron@univ-nantes.fr }


	
	\date{\today}

	
	\maketitle


	\begin{abstract}
		
		This paper is devoted to an inverse Steklov problem for a particular class of $n$-dimensional manifolds having
		the topology of a hollow sphere and equipped with a warped product metric. We prove that the knowledge of the Steklov spectrum determines uniquely
		the associated warping function up to a natural invariance.

		
		\vspace{0.5cm}
		

	\end{abstract}

	\newpage

	\tableofcontents

	\newpage
	
	\section{Introduction}
	
	\subsection{The Calder\'on and Steklov problems.}
	Let $(M,g)$ be a smooth compact manifold of dimension $n \geq 2$ with smooth boundary $\partial M$. We consider the Dirichlet problem
	\begin{equation}
	\label{Schr}
	\left\{
	\begin{aligned}
	& -\Delta_g u=\lambda u \:\:{\rm{in}}\:\:M\\
	& \ \ \  u=\psi\:\:{\rm{on}}\:\:\partial M ,
	\end{aligned}\right.
	\end{equation}
	where $\displaystyle\psi\in H^{1/2}(\partial M)$ and $\lambda\in\mathbb{R}$ is assumed to lie outside the Dirichlet spectrum $\sigma(-\Delta_g)$ of the Laplace-Beltrami operator $-\Delta_g$.
	In a local coordinate system $(x^i)_{i=1, ...n}$, $-\Delta_g$ has the expression
	
	\begin{equation*}
	-\Delta_g=-\sum_{1\le i,j\le n}\frac{1}{\sqrt{|g|}}\partial_i\big(\sqrt{|g|}g^{ij}\partial_j\big),
	\end{equation*}
	
	\medskip
	
	\noindent where we have set $|g|=\det(g_{ij})$ and $(g^{ij})=(g_{ij})^{-1}$.
	
	$\quad$
	
	\noindent If $\lambda \notin \sigma(-\Delta_g)$, the Dirichlet problem (\ref{Schr}) has a unique solution  $u \in H^1(M)$, and we can define the so-called
	\textit{Dirichlet-to-Neumann} (DN) operator as the map
	\begin{equation*}
	\begin{array}{ccccc}
	\Lambda_g(\lambda) & : & H^{1/2}(\partial M) & \to & H^{-1/2}(\partial M) \\
	& & \displaystyle\psi & \mapsto &  \displaystyle\frac{\partial u}{\partial\nu} \big|_{\partial M} \ , \\
	\end{array}
	\end{equation*}
	
	\noindent where $\partial_\nu$ is the unit normal derivative with respect to the unit outer normal vector on $\partial M$. This normal derivative has to be understood in the weak sense by :
	\begin{equation*}
	\forall (\psi,\phi)\in H^{1/2}(\partial M)^2\: :\: \langle\Lambda_g(\lambda)\psi,\phi\rangle=\int_{M}\langle du,dv\rangle_g\,d\mathrm{Vol}_g+\lambda\int_{M}uv\,d\mathrm{Vol}_g,
	\end{equation*}
	where $u$ is the unique weak solution of the Dirichlet problem (\ref{Schr}), and where $v$ is any element of $H^1(M)$ such that $v_{|\partial M} = \phi$. When $\psi$ is sufficiently smooth,
	this definition coincides with the usual one in local coordinates, that is
        \begin{equation} \label{DN-Coord}
        \partial_\nu u = \nu^i \partial_i u.
        \end{equation}
	\medskip
	
	\noindent The anisotropic Calder\'on problem can be initially stated as:{\it{ does the knowledge of the DN map $\Lambda_g(\lambda)$ at a fixed frequency $\lambda$ determine uniquely
	the metric $g$ ?}}
	
	\medskip
	
	\noindent Due to a number of gauge invariances, it is well-known that the answer to the above question is negative in general.
	An observation made by Luc Tartar (\cite{kohn1984identification}, p.2) leads to the equality :
	\begin{equation*}
	\Lambda_g(\lambda)=\Lambda_{\psi^*g}(\lambda),
	\end{equation*}
	
	\medskip
	
	\noindent where $\psi:M\to M$ is any smooth diffeomorphism which is equal to the identity on the boundary, (here $\psi^*g$ is the pullback of $g$ by $\psi$). Moreover, in dimension $n=2$ and for $\lambda=0$, there is one more gauge invariance. Indeed, thanks to the conformal invariance of the Laplacian, for every positive function $c$, we have
	\begin{equation*}
	\Delta_{cg}=\frac{1}{c}\Delta_g.
	\end{equation*}
	
	\noindent
	Consequently, the solutions of the Dirichlet problem (\ref{Schr}) associated to the metrics $g$ and $cg$ are the same when $\lambda=0$. Moreover, if we assume that $c\equiv 1$ on the boundary, the unit outer normal vectors on $\partial M$ are also the same for both metrics. Therefore,
	\begin{equation*}
	\Lambda_{cg}(0)=\Lambda_{g}(0),
	\end{equation*}
	and it is not possible to determine uniquely the metric from the DN map.
		
	\noindent Hence, the appropriate question to adress is :
	
	\medskip
	
		\begin{center}
		\textit{Assume $n\ge 3$ (resp. $n=2$ and $\lambda\ne0$). If $\Lambda_g(\lambda)=\Lambda_{\tilde{g}}(\lambda)$, is there a smooth diffeomorphism $\psi:M\to M$ with $\psi|_{\partial M}=id$ and $\psi^*g=\tilde{g} \:?$}
		\end{center}

	\medskip
	
	\noindent This problem is still largely open. Some special cases have been answered positively (see \cite{salo2013calderon} and \cite{uhlmann2009electrical} for
	a presentation of the latest developments) but the general case seems to be very difficult to tackle. However, the question becomes simpler if we assume that $(M,g)$ and $(M,\tilde{g})$
	belong to the same conformal class. Precisely :
	
	\medskip
	
	\begin{center}
		\textit{For $c\in C^\infty(M)$, assume $\Lambda_g(\lambda)=\Lambda_{cg}(\lambda)$. Then, is there a smooth diffeomorphism $\psi:M\to M$ such that  $\psi|_{\partial M}=id$
		and $\psi^*(g)=cg \:?$}
	\end{center}

\medskip

\noindent Actually, we can precise the above question thanks to the following result due to  Lionheart (\cite{lionheart1997conformal}) : any diffeomorphism $\psi :M\to M$  which satisfies $\psi^*(cg)=g$ and
 $\psi|_{\partial M}=id$ must be the identity. Then, the anisotropic Calder\'on inverse problem within the same conformal class can be replaced by

 \begin{center}
		\textit{If $\Lambda_g(\lambda)=\Lambda_{cg}(\lambda)$, is it true that $c=1 \:?$}
	\end{center}

\medskip
\noindent
Some deep results have been obtained in \cite{ferreira2009limiting} for conformally transversally anisotropic manifolds $(M,g)$ of dimension $n \geq3$, i.e for manifolds
\begin{equation*}
	M\subset\subset \mathbb{R}\times K,\quad g=c(x,y_K)(dx^2+g_K)
\end{equation*}

\medskip

\noindent where $(K,g_K)$ is a smooth compact manifold of dimension $n-1$. Under some geometrical conditions on $K$, such as \textit{simplicity}
(a compact manifold $K$ is said \textit{simple} if any two points of $K$ are connected by a unique geodesic and if its boundary is strictly convex) the conformal
factor $c$ is entirely determined by the DN map at frequency $\lambda=0$.

\medskip
\noindent
In the same way, for a class of manifolds $M= [0,1]\times K$ which have the topology of a cylinder, various results of uniqueness (or non-uniqueness when the Dirichlet data and the Neumann data are measured on disjoint sets)
have been obtained  in (\cite{daude2015non,daude2019hidden}). The proofs make use of separation of variables and introduces a connection between the DN map and the Weyl-Titchmarsh functions
associated to a separated ordinary differential equation corresponding to the horizontal variable of the cylinder. This connection allows to use some nice results from complex analysis.

\medskip
\noindent
In this paper, we are interested in a relative inverse problem within the conformal class of certain warped product manifold $M=[0,1]\times \mathbb{S}^{n-1}$ (where $\mathbb{S}^{n-1}$ is the $(n-1)$-sphere) but under a weaker assumption :

\begin{center}
{\it{does the spectrum of the DN map characterize the conformal factor ?}}
\end{center}

$\quad$

\noindent We recall that $\Lambda_g(\lambda)$ is an elliptic pseudodifferential operator of order 1 and is self-adjoint on $L^2(\partial M, dS_g)$ where $dS_g$ is the metric induced by $g$ on the boundary $\partial M$. Thus, the DN operator $\Lambda_g(\lambda)$  has a discrete spectrum denoted $\sigma(\Lambda_g(\lambda))$ accumulating at infinity
\begin{equation*}
	\sigma(\Lambda_g(\lambda))=\{\lambda_0\leq \lambda_1\le\lambda_2\le...\le\lambda_k\to+\infty\}.
\end{equation*}

\medskip

\noindent
This spectrum is called the \textit{Steklov spectrum} (see \cite{jollivet2014inverse}, p.2). In the particular case where $\lambda=0$, using the Green's formula, we see that the DN operator is positive and we have $\lambda_0 =0$. The properties of the Steklov spectrum are highy sensitive to the smoothness of the boundary $\partial M$. For example (see \cite{girouard2017spectral}): if $\partial M$ is smooth, the eigenvalues satisfy the Weyl formula for the Steklov spectrum:
\begin{equation}
\label{WeylSteklov}
	\lambda_j=2\pi\bigg(\frac{j}{\mathrm{Vol}(\mathbb{B}^{n-1})\mathrm{Vol}(\partial M)}\bigg)^{\frac{1}{n-1}}+O(1)
\end{equation}
A much more refined asymptotic holds if $M$ is a smooth surface, involving the lengths of the connected components of $M$, with a rate of decay of $O(j^{-\infty})$ (see \cite{girouard2019steklov}), but this formula fails for polygons (\cite{girouard2017spectral}). In some particular cases, when the boundary $\partial M$ is just $C^1$, it is also possible to find a one term asymptotic of the Steklov spectrum counted with multiplicity (see \cite{agranovich2006mixed}). For more specific domains with Lipschitz boundary, a recent paper (\cite{girouard2019steklov}) proves a two-term asymptotic formula for cuboids, i.e domains defined, for $n\ge 3,$ as
\begin{equation*}
	M=(-a_1,a_1)\times ... \times (-a_n,a_n)\in\mathbb{R}^n.
\end{equation*}
\noindent The formula for the counting function $N(\lambda)$ of Steklov eigenvalues is the following ($C_1,C_2\in\mathbb{R}$):
\begin{equation*}
N(\lambda)=C_1\mathrm{Vol}_{n-1}(\partial M)\lambda^{n-1}+C_2\mathrm{Vol}_{n-2}(\partial^2 M)\lambda^{n-2}+O(\lambda^{\eta}),
\end{equation*}
where $\partial^2 M$ denotes the union of all the $n-2$ dimensional facets of $M$, $\displaystyle \eta=\frac{2}{3}$ if $n=3$ and $\displaystyle \eta=n-2-\frac{1}{n-1}$ if $n\ge 4$. As a corollary, the authors deduce also that, for a rectangle, the Steklov spectrum determines its side lengths: in other words they are Steklov spectrum invariants. Moreover, it is well-known that some other geometrical quantities of the boundary of a Riemannian manifold surface are Steklov spectrum invariants. Let us mention for instance (see \cite{girouard2019steklov,girouard2017spectral} for details and reference therein):

\begin{enumerate}[label=\alph*), align=left, leftmargin=*, noitemsep]
\item the dimension of the manifold and the volume of its boundary.

\item When $\dim M\ge 3$, the integral of the mean curvature on $\partial M$.

\item When $\dim M=2$, the number and the lengths of the connected components of $\partial M$.
\end{enumerate}

	\subsection{The main result}

In this section, we give the main results of the paper. Let $M=[0,1]\times \mathbb{S}^{n-1}$ be a manifold equipped with the metric
 \begin{equation}
 \label{met}
 g=f(x)(dx^2+g_{\mathbb{S}}),
 \end{equation}
where $\mathbb{S}^{n-1}$ is the unit sphere in $\mathbb{R}^n$, $g_\mathbb{S}$ is the  metric induced by the euclidean metric on $\mathbb{R}^n$, and where
the conformal factor $f$ is a smooth positive function of the variable $x \in [0,1]$ only. Note that the manifold $M$ has two boundaries $\Gamma_0= \{0\} \times \mathbb{S}^{n-1}$ and $\Gamma_1 = \{1\} \times \mathbb{S}^{n-1}$. Thus, the DN map will be shown to have the structure of a matrix operator defined on $H^{1/2}(\mathbb{S}^{n-1})\oplus H^{1/2}(\mathbb{S}^{n-1})$.

\medskip
\noindent
In this setting, we want to answer the following question : does the Steklov spectrum determine uniquely the warping function $f(x)$? This question has been answered positively in \cite{daude2018stability} when $K=(0,1]\times \mathbb{S}^{n-1}$ equipped with the metric (\ref{met}). 
The difference between this case and the one we are studying here is that the boundary of $K$ is connected whereas that of $M$ is made of two connected components. Due to a natural gauge invariance, we emphasize it is hopeless to recover the metric from the spectrum data. Indeed, let $\psi:M\to M$ be a smooth diffeomorphism. We have (\cite{jollivet2014inverse}):

\begin{equation*}
	\Lambda_{\psi^*g}(\lambda)=\varphi^*\circ\Lambda_g(\lambda)\circ{\varphi^*}^{-1} ,
\end{equation*}

\medskip

\noindent where $\varphi:=\psi|_{\partial M}$ and where  $\varphi^* : C^\infty(\partial M)\to C^\infty(\partial M)$ is the application defined by $\varphi^*h:=h\circ \varphi$. As a by-product, one has:
\begin{equation*}
	\sigma(\Lambda_{\psi^*g}(\lambda))=\sigma(\Lambda_g(\lambda)).
\end{equation*}

\medskip

\noindent Thus, if we can find a diffeomorphism $\psi$ preserving the warped product structure of the manifold $M$ given by (\ref{met}), we are able to find a counterexample to uniqueness from the knowledge of the Steklov spectrum. For instance, consider the map

\begin{equation*}
	\psi : (x,y)\in[0,1]\times \mathbb{S}^{n-1} \mapsto (1-x,y)\in[0,1]\times \mathbb{S}^{n-1}.
\end{equation*}

\medskip

\noindent A straightforward computation gives $\displaystyle \psi^*g=f(1-x)(dx^2+g_{\mathbb{S}})$. Thus, the above discussion shows that $\Lambda_{g}(\lambda)$
and $\Lambda_{\psi^*g}(\lambda)$ have the same Steklov spectrum. Now, we can reformulate more precisely our initial question. Let $g$ and $\tilde{g}$ be two Riemannian metrics given by  (\ref{met}) with conformal factor $f(x)$, (resp. $\tilde{f}(x)$). Assume that $\Lambda_g(\lambda)$ and $\Lambda_{\tilde{g}}(\lambda)$ are Steklov isospectral. Then, is it true that :

\begin{center}
	\textit{$f(x) = \tilde{f}(x)$ or $f(x) = \tilde{f}(1-x)$? }
\end{center}

\medskip

\noindent
In what follows, we answer positively this question in dimension $n=2$ with $\lambda\ne 0$, and in dimension $n\ge3$  for any frequency $\lambda$ with an additional hypothesis on the metrics on the boundary $\partial M$.

\medskip

\noindent We choose the sphere $\mathbb{S}^{n-1}$ as the transversal manifold of our cylinder for a purely technical reason. Indeed, we need a precise (in fact an exact) asymptotic of the eigenvalues of the Laplace Beltrami operator $\Delta_{g_\mathbb{S}}$ on $(\mathbb{S}^{n-1},g_\mathbb{S})$  in order to get our uniqueness on the conformal factor $f$. Moreover, it is important to understand that we fix the transversal metric $g_\mathbb{S}$ in our inverse problem. Otherwise, it is known that we could find two non isometric transversal metrics such that the associated Riemannian manifolds are isospectral. This would lead to Steklov isospectral cylinders. Similarly, when $\mathrm{dim}(M)\ge 3$, it is known (see \cite{girouard2019steklov,parzanchevski2013g}) that we can find Steklov isospectral manifolds $M=[0,L]\times K$ and $\tilde{M}=[0,L]\times \tilde{K}$ (with $K$ and $\tilde{K}$ not connected) such that the areas of the connected components of $M$ and $\tilde{M}$ are not the same.
 
\medskip
 
\noindent Our main result is the following:
	
\begin{theorem}\label{theorem1}
Let $M=[0,1]\times \mathbb{S}^{n-1}$ be a smooth Riemanniann manifold equipped with the metric
		\begin{center}
		$ g=f(x)(dx^2+g_\mathbb{S})$ \ ,\ (resp. $\tilde{g}=\tilde{f}(x)(dx^2+g_\mathbb{S}))$,
		\end{center}
and let $\lambda$ be a frequency not belonging to the Dirichlet spectrum of  $-\Delta_g$ and $-\Delta_{\tilde{g}}$ on $M$. Then,
\begin{enumerate}
	 \item For $n=2$ and $\lambda\ne 0$,
	    \begin{equation*}
		\big(\sigma(\Lambda_g(\lambda)) = \sigma(\Lambda_{\tilde{g}}(\lambda))\big)	\Leftrightarrow \big(f=\tilde{f}\quad {\rm{or}}\quad f=\tilde{f}\circ \eta\big)
		\end{equation*}
		where $\eta(x)=1-x$ for all $x \in [0,1]$.
		
	\item For $n\ge 3$, and if moreover
		\begin{equation*}
			f,\tilde{f}\in\mathcal{C}_b := \bigg\{f\in C^\infty([0,1]),\: \bigg|\frac{f'(k)}{f(k)}\bigg|\le \frac{1}{n-2},\:k=0 \ {\rm{and}} \  1 \bigg\},
	     \end{equation*}			
	\begin{equation*}
	\big(\sigma(\Lambda_g(\lambda)) = \sigma(\Lambda_{\tilde{g}}(\lambda))\big)	\Leftrightarrow \big(f=\tilde{f}\quad {\rm{or}}\quad f=\tilde{f}\circ \eta\big)
	\end{equation*}
\end{enumerate}
\end{theorem}

\medskip
\noindent
Let us explain briefly the outline of the proof. In both cases $n=2$ or $n \geq 3$, the proof consists in four steps. We emphasize that in the first three steps, we do not use explicitly that the transversal manifold is the unit sphere $\mathbb{S}^{n-1}$.

$\quad$

\noindent{\it Step 1 :} we follow the same approach as in \cite{daude2015non}. Since the manifold $M$ has the topology of a cylinder and is equipped with a warped product metric, we can use separation of variables and write the solution $u$ of the Dirichlet problem (\ref{Schr}) as
\begin{equation*}
u(x,y)=\sum_{m=0}^{+\infty}u_m(x)Y_m(y),
\end{equation*}
reducing this problem to a countable family of Sturm-Liouville equations with boundary conditions :

\begin{equation}
\label{SL_intr}
\displaystyle\: \left\{
\begin{aligned}
& -v^{''}_m+qv_m=-\mu_mv_m,\quad {\rm{on}}\:\:]0,1[\\
&v_m(0)=f^{\frac{n-2}{4}}(0)\psi_m^0,\:\:v_m(1)=f^{\frac{n-2}{4}}(1)\psi_m^1, \end{aligned}\right.\quad\quad \forall m\in\mathbb{N}
\end{equation}
\noindent where \begin{center}
	$\displaystyle q=\frac{(f^{\frac{n-2}{4}})''}{f^{\frac{n-2}{4}}}-\lambda f\quad$ and $\quad v_m=f^{\frac{n-2}{4}}u_m,\: \forall m\in\mathbb{N}$.
\end{center}

\noindent The sequence $(Y_m)$ refers to an orthonormal basis of eigenvectors of the Laplace-Beltrami operator $-\Delta_{g_\mathbb{S}}$ on the unit sphere.

$\quad$

\noindent{\it Step 2 :} we use the above decomposition to write the DN operator as an infinite matrix which is block diagonal. More precisely, let us introduce the basis  $\mathfrak{B}$=$(\{e^1_m,e^2_m\})_{m \in \mathbb{N}}$, where $e^1_m=(Y_m,0)$ and $e^2_m=(0,Y_m)$. For each $m\in\mathbb{N}$, we denote $\Lambda_g^m(\lambda)$ the restricted operator of $\Lambda_g(\lambda)$ on the subspace spanned by $\{e^1_m,e^2_m\}$. We can write $\Lambda_g(\lambda)$ in the basis $\mathfrak{B}$ as the infinite matrix:

\medskip

\begin{equation*}
	[\Lambda_g]_{\mathfrak{B}}=\begin{pmatrix}
	&\Lambda_g^1(\lambda)&&0&&0&\cdots&\\
	&0&&\Lambda_g^2(\lambda)&&0&\cdots&\\
	&0&&0&&\Lambda_g^3(\lambda)&\cdots&\\
	&\vdots&&\vdots&&\vdots&\ddots
	\end{pmatrix}
\end{equation*}

\medskip

\noindent
Each $(2,2)$ matrix $\Lambda_g^m(\lambda)$ has a simple interpretation involving the so-called Weyl-Titchmarsh theory associated to the Sturm-Liouville equation (\ref{SL_intr}). More precisely, if we denote
\begin{equation*}
\sigma(-\Delta_{g_\mathbb{S}})=\{0=\mu_0<\mu_1\le\mu_2\le...\le \mu_m\le...\to+\infty\},
\end{equation*}
the spectrum of $-\Delta_{g_\mathbb{S}}$ and $h:=f^{n-2}$, we shall see that :
\begin{equation*}
\Lambda_g^m(\lambda)=\begin{pmatrix}
-\frac{M(\mu_m)}{\sqrt{f(0)}}+\frac{1}{4\sqrt{f(0)}}\frac{h'(0)}{h(0)}&-\frac{1}{\sqrt{f(0)}}\frac{h^{1/4}(1)}{h^{1/4}(0)}\frac{1}{\Delta(\mu_m)}\\
-\frac{1}{\sqrt{f(1)}}\frac{h^{1/4}(0)}{h^{1/4}(1)}\frac{1}{\Delta(\mu_m)}&-\frac{N(\mu_m)}{\sqrt{f(1)}}-\frac{1}{4\sqrt{f(1)}}\frac{h'(1)}{h(1)}
\end{pmatrix}.
\end{equation*}

\noindent where $M(z)$, $N(z)$ are the Weyl-Titchmarsh functions and $\Delta(z)$ is the characteristic function associated to (\ref{SL_intr}). In particular, the trace (respectively the determinant) of these operators $\Lambda_g^m(\lambda)$ are  meromorphic functions evaluated in $\mu_m$. Moreover, one can prove that the Steklov Spectrum is made of two subsequences $\big(\lambda^-(\mu_m)\big)_{m\in\mathbb{N}}$ and $\big(\lambda^+(\mu_m)\big)_{m\in\mathbb{N}}$ satisfying the following asymptotic expansion 
\begin{equation*}
	\left\{ \begin{aligned}
	& \lambda^-(\mu_m)=-\frac{N(\mu_m)}{\sqrt{f(1)}}-\frac{1}{4\sqrt{f(1)}}\frac{h'(1)}{h(1)}+O\big(\sqrt{\mu_m}e^{-\sqrt{\mu_m}} \big) \\
	&\lambda^+(\mu_m)=-\frac{M(\mu_m)}{\sqrt{f(0)}}+\frac{1}{4\sqrt{f(0)}}\frac{h'(0)}{h(0)}+O\big(\sqrt{\mu_m}e^{-\sqrt{\mu_m}} \big)
	\end{aligned}\right.
\end{equation*}

$\quad$

\noindent {\it Step 3 :} we prove that the knowledge of the trace and the determinant of all restricted operators $\Lambda_g^m(\lambda)$, for $m$ large enough, characterizes $f$ up to an involution: there is $m_0\in\mathbb{N}$ such that :

\begin{center}
	$\big($Tr\ $\Lambda_g^m(\lambda)=$ Tr\ $\Lambda_{\tilde{g}}^m(\lambda)$ and $\det \Lambda_g^m(\lambda)=\det\Lambda_{\tilde{g}}^m(\lambda)$ for all $m \geq m_0$ $\big)$ 
	
	$\Leftrightarrow$ ($f=\tilde{f}$ or $f=\tilde{f}\circ\eta$.)
\end{center}

\noindent
Note that the third step requires to know the asymptotic expansion of the Steklov spectrum.

$\quad$

\noindent{\it Step 4} : in this step, we use explicitely that the transversal manifold is the unit sphere $\mathbb{S}^{n-1}$ equipped with the metric induced by the euclidean metric on $\mathbb{R}^{n}$. From the equality between the sets $\sigma\big(\Lambda_g(\lambda)\big)=\big(\lambda^\pm(\mu_m)\big)_{m\in\mathbb{N}}$ and $\sigma\big(\Lambda_{\tilde{g}}(\lambda)\big)=\big(\tilde{\lambda}^\pm(\mu_m)\big)_{m\in\mathbb{N}}$, we want to deduce the equalities
\begin{equation*}
	\left\{\begin{aligned}
	&\lambda^-(\mu_m)= \tilde{\lambda}^-(\mu_m)\\
	&\lambda^-(\mu_m)= \tilde{\lambda}^+(\mu_m)
	\end{aligned}\right.\quad\mathrm{or}\quad 	\left\{\begin{aligned}
	&\lambda^-(\mu_m)= \tilde{\lambda}^+(\mu_m)\\
	&\lambda^+(\mu_m)= \tilde{\lambda}^-(\mu_m),
	\end{aligned}\right.
\end{equation*}
\noindent for integers $m$ belonging to a set $\mathcal{L}$ satisfying the Müntz conditions $\displaystyle \sum_{m\in\mathcal{L}}\frac{1}{m}=+\infty$. Clearly, for every $\lambda^\pm(\mu_m)$ in $\sigma\big(\Lambda_g(\lambda)\big)$, there is $\tilde{\lambda}^-(\mu_\ell)$ or $\tilde{\lambda}^+(\mu_\ell)$ in $\sigma\big(\Lambda_{\tilde{g}}(\lambda)\big)$ such that 
\begin{center}
	$\lambda^\pm(\mu_m)=\tilde{\lambda}^-(\mu_\ell)\quad$ or $\quad\lambda^\pm(\mu_m)=\tilde{\lambda}^+(\mu_\ell)$.
\end{center}
\noindent The assumption $\displaystyle \bigg|\frac{f'(k)}{f(k)}\bigg|\le \frac{1}{n-2}$ for $k\in\{0,1\}$ is used here to ensure that $m=\ell$ if is $m$ is large enough. This step leads to distinguish the cases $f(0)=f(1)$ and $f(0)\ne f(1)$. Then we prove that $\mathrm{Tr}\ \Lambda_g^m(\lambda)= \mathrm{Tr}\ \Lambda_{\tilde{g}}^m (\lambda)$ and $\det \Lambda_g^m(\lambda)=\det\ \Lambda_{\tilde{g}}^m(\lambda)$ for all $m$ large enough. The result follows from the Step 3.

\section{Reduction to ordinary differential equations}

\noindent In this section, $(K,g_K)$ is an arbitrary closed manifold of dimension $n-1$ and $M=[0,1]\times K$ is equipped with the metric $g=f(x)(dx^2+g_K)$.

\subsection{The separation of variables}

\medskip

The boundary $\partial M$ of the manifold $M$ has two distinct connected components
\begin{center}
	$\Gamma_0=\{0\}\times K$ and $\Gamma_1=\{1\}\times K$,
\end{center}
so we can decompose $H^{1}(\partial M)$ as the direct sum :
\begin{center}
	$H^{1/2}(\partial M)=H^{1/2}(\Gamma_0) \bigoplus H^{1/2}(\Gamma_1)$.
\end{center}

\medskip

\noindent
Each  element $\psi$ of $H^{1/2}(\partial M)$ can be written as
\begin{center}
$\displaystyle \psi=\begin{pmatrix}
\psi^0\\
\psi^1
\end{pmatrix},\quad\quad$ $\psi^0\in H^{1/2}(\Gamma_0)$ and $\psi^1\in H^{1/2}(\Gamma_1)$.
\end{center}

$\quad$

\noindent The Laplacian $-\Delta_{g_K}$ is a self-adjoint operator on $L^2(K)$ and has pure point spectrum $(\mu_m)$, with $\mu_0=0<\mu_1\le\mu_2\le...\: \mu_m\to+\infty$. We denote $(Y_m)_{m\in\mathbb{N}}$  the associated orthonormal Hilbert basis of the eigenvectors.

\medskip
\noindent
Now, we decompose $\psi^0$ and $\psi^1$ as:

\medskip

\begin{equation*}
\psi^0=\sum_{m\in\mathbb{N}}\psi_m^0 Y_m,\quad\quad \psi^1=\sum_{m\in\mathbb{N}}\psi_m^1 Y_m.
\end{equation*}

\medskip

\noindent
Then, we are looking for the unique solution $u(x,y)$ of the Dirichlet problem (\ref{Schr}) in the form:
\begin{equation*}
u(x,y)=\sum_{m=0}^{+\infty}u_m(x)Y_m(y).
\end{equation*}

\begin{proposition}
	\label{equiveqdiff}
The equation (\ref{Schr}) is equivalent to the following countable system of Sturm-Liouville equations :

\begin{equation}
\label{S-T}
	\displaystyle\: \left\{
	\begin{aligned}
	& -v^{''}_m+qv_m=-\mu_mv_m,\quad {\rm{on}}\:\:]0,1[\\
	&v_m(0)=f^{\frac{n-2}{4}}(0)\psi_m^0,\:\:v_m(1)=f^{\frac{n-2}{4}}(1)\psi_m^1, \end{aligned}\right.\quad\quad \forall m\in\mathbb{N},
\end{equation}
\noindent where \begin{center}
	$\displaystyle q=\frac{(f^{\frac{n-2}{4}})''}{f^{\frac{n-2}{4}}}-\lambda f\quad$ and $\quad v_m=f^{\frac{n-2}{4}}u_m,\: \forall m\in\mathbb{N}$.
\end{center}
\end{proposition}

\begin{proof}
The transformation law for the Laplacian operator under conformal change of metric gives
\begin{equation*}
\begin{aligned}
-\Delta_g=f^{-\frac{n+2}{4}}(-\Delta_{g_0}+q_f)f^{\frac{n-2}{4}} ,
\end{aligned}
\end{equation*}
where $g_0=dx^2+g_{K}$ and $\displaystyle q_f=:\frac{(f^\frac{{n-2}}{4})''}{f^{\frac{n-2}{4}}}$. Thus, we have :
\begin{equation*}
\begin{aligned}
	-\Delta_g u(x,y)=\lambda u(x,y)&\Leftrightarrow f^{-\frac{n+2}{4}}\big(-\frac{\partial^2}{\partial x^2}-\Delta_{g_K}+q_f\big)f^{\frac{n-2}{4}}u(x,y)=\lambda u(x,y)\\
	&\Leftrightarrow \big(-\frac{\partial^2}{\partial x^2}-\Delta_{g_K}+q_f\big)v(x,y)=\lambda f(x)v(x,y)\\
\end{aligned}	
\end{equation*}

\noindent thanks to the change of variable $\displaystyle v=f^{\frac{n-2}{4}}u$. Writing $v$ as

\begin{equation*}
	v(x,y)=\sum_{m=0}^{+\infty}v_m(x)Y_m(y),
\end{equation*}
\noindent we get :
\begin{equation*}
\begin{aligned}
-\Delta_g u(x,y)=\lambda u(x,y)&\Leftrightarrow \sum_{m=0}^{+\infty}\big(-\frac{\partial^2}{\partial x^2}-\Delta_{g_K}+q_f\big)v_m(x)Y_m=\sum_{m=0}^{+\infty}\lambda f(x)v_m(x)Y_m\\
&\Leftrightarrow \sum_{m=0}^{+\infty}\big(-v''_m(x)+\mu_mv_m(x)+q_f(x)v_m(x)\big)Y_m\\
&\qquad\qquad\qquad\qquad\qquad\qquad\qquad\qquad=\sum_{m=0}^{+\infty}\lambda f(x)v_m(x)Y_m\\
&\Leftrightarrow \forall m\in\mathbb{N},\:
 -v^{''}_m(x)+q(x)v_m(x)=-\mu_mv_m(x),\\
\end{aligned}
\end{equation*}
where $\displaystyle q=:q_f-\lambda f$. Finally, $v_m$ satisfies the above boundary conditions.
\end{proof}

\medskip

\noindent
It turns out that this family of Sturm-Liouville equations fits into the so-called Weyl-Titchmarsh theory which we recall in the next section for the convenience of the reader.

\subsection{The Weyl-Titchmarsh functions}

\medskip

\noindent
Consider the differential equation:
\begin{equation}
\label{eqdif}
	-u''+qu=-zu,\quad z\in\mathbb{C}.
\end{equation}
Let $\{c_0,s_0\}$ and $\{c_1,s_1\}$ be the two fundamental systems of solutions of (\ref{eqdif}) with boundary conditions
\begin{equation}
\label{solfond}
	\left\{
	\begin{aligned}
	& c_0(0)=1,\:c_0'(0)=0\\
	& c_1(1)=1,\:c_1'(1)=0 \end{aligned}\right.\quad{\rm{and}}\quad \left\{
	\begin{aligned}
	& s_0(0)=0,\:s_0'(0)=1\\
	& s_1(1)=0,\:s_1'(1)=1. \end{aligned}\right.
\end{equation}

$\quad$

\noindent
Since the wronskian $W(f,g) = fg'-f'g$ of two solutions $f$ and $g$ of (\ref{eqdif}) depends only  on the parameter $z$, we can define the following holomorphic functions:

\medskip

\begin{definition}
The \textit{characteristic function} $\Delta(z)$ of the equation (\ref{eqdif}) is defined for every $z\in\mathbb{C}$ by :
\begin{equation*}
\Delta(z):=W(s_0,s_1)=s_0(1)=-s_1(0).
\end{equation*}
We also set  $D(z):=W(c_0,s_1)=c_0(1)$ and $E(z):=W(c_1,s_0)=c_1(0)$.
\end{definition}

\noindent
Now, we recall the three following results, (see for instance \cite{daude2015non} and \cite{poschel1987inverse} for details).

\medskip

\begin{proposition}
	\label{estim}
Set $\Pi^+=\{z\in\mathbb{C},\:\Re(z)>0\}$ the half-right plane of the complex plane. The functions $\Delta$, $D$ and $E$ are analytic on $\mathbb{C}$ and have on $\Pi^+$ the following asymptotics as $|z|\to +\infty$:

\begin{center}
$\displaystyle {\Delta(z)=\frac{\sinh(\sqrt{z})}{\sqrt{z}}\:+\:O\bigg(\frac{e^{|\Re(\sqrt{z})|}}{z}\bigg)},\quad {D(z)=\cosh(\sqrt{z})\:+\:O\bigg(\frac{e^{|\Re(\sqrt{z})|}}{\sqrt{z}}\bigg)}$,\\ \quad $\displaystyle{E(z)=\cosh(\sqrt{z})+O\bigg(\frac{e^{|\Re(\sqrt{z})|}}{\sqrt{z}}\bigg)}$.
\end{center}
where $\sqrt{z}$ is the principal square root of $z$. In particular, $\Delta(z)$, $D(z)$ and $E(z)$ are entire functions of order $\displaystyle \frac{1}{2}$.
\end{proposition}

\medskip

\begin{proposition}
	\label{racDelt}
	The roots $(\alpha_j)$ of the function $z\mapsto \Delta(z)$ are real and simple. they are the opposite of the eigenvalues of the operator $\displaystyle -\frac{d^2}{dx^2}+q=:H$ on $L^2\big((0,1),dx\big)$ with Dirichlet boundary conditions.
\end{proposition}

\medskip

\begin{proposition}
	\label{Prodinf}
For all $z\in\mathbb{C}$, $\Delta(z)$ can be written as an infinite product. There is a constant $C\in\mathbb{R}$ such that :

\begin{equation*}
\Delta(z)=C\prod_{k=0}^{\infty}\bigg(1-\frac{z}{\alpha_k}\bigg).
\end{equation*}
\end{proposition}

\medskip

\begin{proof}
As a consequence of Proposition \ref{estim} and Hadamard's factorization theorem, we can write $\Delta(z)$ as the infinite product :
\begin{equation*}
\Delta(z)=Cz^p\prod_{k=0}^{\infty}\bigg(1-\frac{z}{\alpha_k}\bigg),
\end{equation*}

\noindent with $p\in\{0,1\}$. In order to prove that $z=0$ is not a root of $\Delta$, we use the same argument as in \cite{daude2015non} (Remark 3.1 p.19). If $\Delta(0)=0$, it follows from Proposition \ref{racDelt}, that there is an eigenfunction $u_0$ associated to the eigenvalue $0$ for the operator $\displaystyle \displaystyle H= -\frac{d^2}{dx^2}+q$. But, from Proposition \ref{equiveqdiff}, the function $u:=u_0Y_0$ is then a nontrivial solution of the Dirichlet problem :
\begin{equation*}
\left\{
\begin{aligned}
& -\Delta_g u=\lambda u \:\:{\rm{in}}\:\:M\\
& u=0\:\:{\rm{on}}\:\:\partial M, \end{aligned}\right.
\end{equation*}

\medskip

\noindent
which is not possible since  $\lambda\notin\sigma(-\Delta_g)$.
\end{proof}

\medskip

\begin{remark}
As a by-product, we see that $\Delta (z)$ is uniquely determined by its (simple) roots (up to a multiplicative constant).
\end{remark}

$\quad$

\noindent
Now, consider the Weyl-Titchmarsh solutions $\psi$ and $\phi$ of (\ref{eqdif}) having the form :
\begin{equation*}
	\psi(x)=c_0(x)+M(z)s_0(x),\quad \phi(x)=c_1(x)-N(z)s_1(x)
\end{equation*}
and satisfying the Dirichlet boundary condition at $x=1$ and $x=0$ respectively. $M(z)$, (resp. $N(z)$) are called the \textit{Weyl-Titchmarsh functions} associated to (\ref{eqdif}) and, using Wronskian identities, we have:

\begin{proposition}
The Weyl-Titchmarsh functions $M$ and $N$ can be written as :
\begin{equation*}
	\forall z\in\mathbb{C},\:\:M(z)=-\frac{D(z)}{\Delta(z)},\quad N(z)=-\frac{E(z)}{\Delta(z)}.
\end{equation*}
\end{proposition}

\begin{proof}
\noindent Since $\psi(1)=0$ and $\phi(0)=0$, we have :

\begin{center}
$\displaystyle M(z)=-\frac{c_0(1,z)}{s_0(1,z)}=-\frac{D(z)}{\Delta(z)},\quad$ $\displaystyle N(z)=\frac{c_1(0,z)}{s_1(0,z)}=-\frac{E(z)}{\Delta(z)}$.
\end{center}
\end{proof}

\noindent
The four previous propositions are the key points to solve our uniqueness result. We will see that the Steklov spectrum can be expressed in terms of the Weyl-Titchmarsh and characteristic functions defined above. We will take advantage of the holomorphic properties of $\Delta(z)$, $D(z)$ and $E(z)$  and we will use to the Nevanlinna theorem, (see the next section for details).

\medskip

\subsection{Link between the DN map and the Weyl-Titchmarsh functions}

\medskip
\noindent
First, we remark that, thanks to separation of variables, $\Lambda_g(\lambda)$ leaves invariant each subspace spanned by $\big\{(Y_m,0),(0,Y_m)\big\}$. Indeed, if $u$ is the solution of (\ref{Schr}), we have for each $\psi\in H^{1/2}(\partial M)$:
\begin{equation*}
\Lambda_g(\lambda)\psi=\Lambda_g(\lambda)\begin{pmatrix}
\psi^0\\
\psi^1
\end{pmatrix}=\begin{pmatrix}
(\partial_\nu u)_{{|{\Gamma_0}}}\\
(\partial_\nu u)_{{|{\Gamma_1}}}
\end{pmatrix}=\begin{pmatrix}
-\frac{1}{\sqrt{f(0)}}(\partial_x u)_{{|{x=0}}}\\
\frac{1}{\sqrt{f(1)}}(\partial_x u)_{{|{x=1}}}
\end{pmatrix}
\end{equation*}
Consequently, for every $m\in\mathbb{N}$ :
\begin{equation*}
	\Lambda_g(\lambda)\begin{pmatrix}
	\psi^0_m\\
	\psi^1_m
	\end{pmatrix}\otimes Y_m=\begin{pmatrix}
	-\frac{1}{\sqrt{f(0)}}u_m'(0)\\
	\frac{1}{\sqrt{f(1)}}u'_m(1)
	\end{pmatrix}\otimes Y_m.
\end{equation*}

\medskip

\noindent
Its restriction on each space spanned by $(1,0)\otimes Y_m$ and $(0,1)\otimes Y_m$ is denoted $\Lambda_g^m(\lambda)$. We can write $\Lambda_g^m(\lambda)$
the $2\times2$ matrix

\noindent Set :
\begin{equation*}
   \begin{pmatrix}
	L^m(\lambda)&T^m_R(\lambda)\\
	T^m_L(\lambda)&R^m(\lambda)
	\end{pmatrix}
\end{equation*}
and we have by definition:

\begin{equation*}
	\begin{pmatrix}
	L^m(\lambda)&T^m_R(\lambda)\\
	T^m_L(\lambda)&R^m(_\lambda)
	\end{pmatrix}\begin{pmatrix}
	\psi_m^0\\
	\psi_m^1
	\end{pmatrix}=\begin{pmatrix}
	-\frac{1}{\sqrt{f(0)}}u_m'(0)\\
	\frac{1}{\sqrt{f(1)}}u'_m(1)
	\end{pmatrix}.
\end{equation*}

$\quad$

\medskip

\noindent
The full Steklov spectrum is then equal to the union of the eigenvalues of each operator $\Lambda_g^m(\lambda)$. In the next Proposition, we express the restricted
operator $\Lambda_g^m(\lambda)$ in terms of the Weyl-Titchmarsh functions.

\medskip

\begin{proposition} \label{Mat} For all $m\in\mathbb{N}$, we have :
	
	\begin{equation*}
		\Lambda_g^m(\lambda)=\begin{pmatrix}
	-\frac{M(\mu_m)}{\sqrt{f(0)}}+\frac{\ln'(h)(0)}{4\sqrt{f(0)}}&-\frac{1}{\sqrt{f(0)}}\frac{h^{1/4}(1)}{h^{1/4}(0)}\frac{1}{\Delta(\mu_m)}\\
	-\frac{1}{\sqrt{f(1)}}\frac{h^{1/4}(0)}{h^{1/4}(1)}\frac{1}{\Delta(\mu_m)}&-\frac{N(\mu_m)}{\sqrt{f(1)}}-\frac{\ln'(h)(1)}{4\sqrt{f(1)}}
	\end{pmatrix}
	\end{equation*}
	where $h=:f^{n-2}$
\end{proposition}

\medskip

\begin{proof}

\noindent Recall that $v_m$ is the $m$-th Fourier coefficient of $v=h^{1/4}u$ with respect to $Y_m$. We have :

\begin{equation*}
\begin{aligned}
\Lambda_g^m(\lambda)\begin{pmatrix}
\psi_m^0\\
\psi_m^1
\end{pmatrix}&=\begin{pmatrix}
-\frac{1}{\sqrt{f(0)}}u'_m(0)\\
\frac{1}{\sqrt{f(1)}}u'_m(1)
\end{pmatrix}\\
&=\begin{pmatrix}
-\frac{v_m'(0)}{h^{1/4}(0)\sqrt{f(0)}}+\frac{1}{4\sqrt{f(0)}}\frac{h'(0)}{h(0)}u_m(0)\\
\frac{v_m'(1)}{h^{1/4}(1)\sqrt{f(1)}}-\frac{1}{4\sqrt{f(1)}}\frac{h'(1)}{h(1)}u_m(1)
\end{pmatrix}\\
&=\begin{pmatrix}
-\frac{v_m'(0)}{h^{1/4}(0)\sqrt{f(0)}}\\
\frac{v_m'(1)}{h^{1/4}(1)\sqrt{f(1)}}
\end{pmatrix}+\begin{pmatrix}
\frac{1}{4\sqrt{f(0)}}\frac{h'(0)}{h(0)}\psi_m^0\\
-\frac{1}{4\sqrt{f(1)}}\frac{h'(1)}{h(1)}\psi_m^1
\end{pmatrix}
\end{aligned}
\end{equation*}

\medskip

 \noindent It remains to find a simple expression of the first term of the (RHS). Since $v_m$ is a solution of (\ref{S-T}), it can be written as a linear combination of the fundamental solutions $\big\{c_0,s_0\big\}$ or $\big\{c_1,s_1\big\}$ defined in $(\ref{solfond})$, i.e there exists $(\alpha,\beta,\gamma,\delta)\in\mathbb{C}^4$ such that :
\begin{equation*}
	v_m=\alpha c_0+\beta s_0=\gamma c_1+\delta s_1.
\end{equation*}

\medskip

\noindent Consequently:
\begin{equation*}
\begin{pmatrix}
	v_m(0)\\v_m(1)
	\end{pmatrix}=\begin{pmatrix}
	\alpha\\
	\gamma
	\end{pmatrix}=\begin{pmatrix}
	\gamma c_1(0)+\delta s_1(0)\\
	\alpha c_0(1)+\beta s_0(1)
	\end{pmatrix}
\end{equation*}

\medskip

\noindent The second equality can be rewritten as:
\begin{equation*}
	\begin{pmatrix}
	\alpha-\gamma c_1(0)\\ \gamma-\alpha c_0(1)
	\end{pmatrix}=\begin{pmatrix}
	\delta s_1(0)\\ \beta s_0(1)
	\end{pmatrix},
\end{equation*}
which is equivalent to :
\begin{equation*}
	\begin{pmatrix}
	1&-c_1(0)\\
	-c_0(1)&1
	\end{pmatrix}\begin{pmatrix}
	\alpha\\
	\gamma
	\end{pmatrix}=\begin{pmatrix}
	\delta s_1(0)\\
	\beta s_0(1)
	\end{pmatrix}.
\end{equation*}

$\quad$

\medskip

\noindent As
\begin{equation*}
	\begin{pmatrix}
	\alpha\\ \gamma
	\end{pmatrix}=\begin{pmatrix}
	h^{1/4}(0)u_m(0)\\ h^{1/4}(1)u_m(1)
	\end{pmatrix}=\begin{pmatrix}
	h^{1/4}(0)\psi_m^0\\ h^{1/4}(1)\psi_m^1
	\end{pmatrix}
\end{equation*}

\noindent we obtain:

\medskip

\begin{equation*}
\begin{pmatrix}
\frac{h^{1/4}(0)}{s_1(0)}&-\frac{h^{1/4}(1)c_1(0)}{s_1(0)}\\
-\frac{h^{1/4}(0)c_0(1)}{s_0(1)}&\frac{h^{1/4}(1)}{s_0(1)}
\end{pmatrix}\begin{pmatrix}
\psi_m^0\\
\psi_m^1
\end{pmatrix}=\begin{pmatrix}
\delta\\
\beta
\end{pmatrix}.
\end{equation*}

\medskip

$\quad$

\noindent But $\delta=v'_m(1)$ and $\beta=v'_m(0)$. Thus:

\medskip

\begin{equation*}
	\begin{aligned}
	\Lambda_g^m(\lambda)\begin{pmatrix}
	\psi_m^0\\
	\psi_m^1
	\end{pmatrix}&=\begin{pmatrix}
	-\frac{v_m'(0)}{h^{1/4}(0)\sqrt{f(0)}}\\
	\frac{v_m'(1)}{h^{1/4}(1)\sqrt{f(1)}}
	\end{pmatrix}+\begin{pmatrix}
	\frac{1}{4\sqrt{f(0)}}\frac{h'(0)}{h(0)}\psi_m^0\\
	-\frac{1}{4\sqrt{f(1)}}\frac{h'(1)}{h(1)}\psi_m^1
	\end{pmatrix}\\
	&=\begin{pmatrix}\frac{1}{\sqrt{f(0)}}\frac{c_0(1)}{s_0(1)}&-\frac{1}{\sqrt{f(0)}}\frac{h^{1/4}(1)}{h^{1/4}(0)s_0(1)}\\
	\frac{1}{\sqrt{f(1)}}\frac{h^{1/4}(0)}{h^{1/4}(1)s_1(0)}&-\frac{1}{\sqrt{f(1)}}\frac{c_1(0)}{s_1(0)}
	\end{pmatrix}\begin{pmatrix}
	\psi_m^0\\\psi_m^1
	\end{pmatrix}+\begin{pmatrix}
	\frac{1}{4\sqrt{f(0)}}\frac{h'(0)}{h(0)}\psi_m^0\\
	-\frac{1}{4\sqrt{f(1)}}\frac{h'(1)}{h(1)}\psi_m^1
	\end{pmatrix}\\
&=\begin{pmatrix}
\frac{1}{\sqrt{f(0)}}\frac{c_0(1)}{s_0(1)}+\frac{1}{4\sqrt{f(0)}}\frac{h'(0)}{h(0)}&-\frac{1}{\sqrt{f(0)}}\frac{h^{1/4}(1)}{h^{1/4}(0)s_0(1)}\\
\frac{1}{\sqrt{f(1)}}\frac{h^{1/4}(0)}{h^{1/4}(1)s_1(0)}&-\frac{1}{\sqrt{f(1)}}\frac{c_1(0)}{s_1(0)}-\frac{1}{4\sqrt{f(1)}}\frac{h'(1)}{h(1)}
\end{pmatrix}\begin{pmatrix}
\psi_m^0\\
\psi_m^1
\end{pmatrix}\\
\end{aligned}
\end{equation*}

\noindent Recalling that:

\begin{equation*}
	M(\mu_m)=-\frac{c_0(1)}{s_0(1)},\quad N(\mu_m)=\frac{c_1(0)}{s_1(0)}\quad {\rm{et}}\quad \Delta(\mu_m)=-s_1(0)=s_0(1),
\end{equation*}

\noindent we get finally:

\begin{equation*}
	\Lambda_g^m(\lambda)=\begin{pmatrix}
	-\frac{M(\mu_m)}{\sqrt{f(0)}}+\frac{1}{4\sqrt{f(0)}}\frac{h'(0)}{h(0)}&-\frac{1}{\sqrt{f(0)}}\frac{h^{1/4}(1)}{h^{1/4}(0)}\frac{1}{\Delta(\mu_m)}\\
	-\frac{1}{\sqrt{f(1)}}\frac{h^{1/4}(0)}{h^{1/4}(1)}\frac{1}{\Delta(\mu_m)}&-\frac{N(\mu_m)}{\sqrt{f(1)}}-\frac{1}{4\sqrt{f(1)}}\frac{h'(1)}{h(1)}
	\end{pmatrix}.
\end{equation*}
\end{proof}

\section{A characterisation by the trace and the determinant}

\noindent In this section, $(K,g_K)$ is still an arbitrary closed manifold of dimension $n-1$ and $M=[0,1]\times K$ is equipped with the metric $g=f(x)(dx^2+g_K)$. We have the following result:

\medskip

\begin{proposition}
	\label{trdet}
	Assume that, for every $m\in\mathbb{N}$, we have :
	
	\begin{center}
	$\det(\Lambda_g^m(\lambda))=\det(\Lambda^m_{\tilde{g}}(\lambda))$ and \rm{Tr}$(\Lambda_g^m(\lambda))$=Tr$(\Lambda^m_{\tilde{g}}(\lambda))$.
	\end{center}Then  :
	\begin{equation*}
	f=\tilde{f}\quad {\rm{or}}\quad f=\tilde{f}\circ \eta
	\end{equation*}
	
	\noindent where, for all $x\in [0,1],\:\eta(x)=1-x$.
\end{proposition}

\medskip

\begin{remark}
	\label{Remarque_importante_1}
	This Proposition is still true if the equalities about the trace and the determinant of $\Lambda_g^m(\lambda)$ are are satisfied for $m\ge m_0$, with $m_0\in\mathbb{N}$.
\end{remark}

\medskip

\noindent In order to prove this proposition, let us calculate the eigenvalues of the operator $\Lambda_g^m(\lambda)$ and their asymptotics.

\medskip

\begin{lemma}
	\label{vp}
	$\Lambda_g^m(\lambda)$ has two eigenvalues $\lambda^-(\mu_m)$ and $\lambda^+(\mu_m)$ whose asymptotics are given by :
	\begin{equation*}
	\left\{
	\begin{aligned}
	&\lambda^-(\mu_m)\underset{m\to\infty}{=}\frac{\sqrt{\mu_m}}{\sqrt{f(1)}}-\frac{\ln(h)'(1)}{4\sqrt{f(1)}}+O\bigg(\frac{1}{\sqrt{\mu_m}}\bigg)	 \\
	&\lambda^+(\mu_m)\underset{m\to\infty}{=}\frac{\sqrt{\mu_m}}{\sqrt{f(0)}}+\frac{\ln(h)'(0)}{4\sqrt{f(0)}}+O\bigg(\frac{1}{\sqrt{\mu_m}}\bigg).
	\end{aligned}\right.
	\end{equation*}
\end{lemma}

\medskip

\begin{remark}
	If $f(0)<f(1)$, we have $\lambda^-(\mu_m)<\lambda^+(\mu_m)$. This assumption will be made, without loss of generality, each time that $f(0)\ne f(1)$. The notations $\lambda^-(\mu_m)$ and $\lambda^+(\mu_m)$ refer to that choice.
\end{remark}

$\quad$

\begin{remark}
	The Weyl's law for the eigenvalues of the Laplace-Beltrami operator gives the following asymptotic :
	\begin{equation*}
		\mu_m = 4\pi^2 \bigg(\mathrm{Vol}(\mathbb{B}^{n-1})\mathrm{Vol}(K)\bigg)^{-\frac{2}{n-1}}m^{\frac{2}{n-1}} + O(1),
	\end{equation*} 
so, by replacing it in Lemma \ref{vp}, one has :
\begin{equation*}
	\begin{aligned}
	\lambda^-(\mu_m)&\underset{m\to\infty}{=}\frac{\sqrt{\mu_m}}{\sqrt{f(1)}}+O(1)\\
	&=\frac{2\pi}{\sqrt{f(1)}}\bigg(\mathrm{Vol}(\mathbb{B}^{n-1})\mathrm{Vol}(K)\bigg)^{-\frac{1}{n-1}}m^{\frac{1}{n-1}}+O(1).
	\end{aligned}
\end{equation*}
\noindent The boundary component $\Gamma_1$ consists in copy of $K$ equipped with the metric $\gamma_1=f(1)g_K$. It follows that we have $\displaystyle \mathrm{Vol}(\Gamma_1)=\int_{\Gamma_1}\,\mathrm{dVol}_{\gamma_1}=f(1)^{\frac{n-1}{2}}\mathrm{Vol}(K)$, hence $\displaystyle \frac{1}{\sqrt{f(1)}}=\frac{\mathrm{Vol}(K)^{\frac{1}{n-1}}}{\mathrm{Vol}(\Gamma_1)^{\frac{1}{n-1}}}$. Consequently :
\begin{equation*}
		\lambda^-(\mu_m)= 2\pi \bigg(\frac{m}{\mathrm{Vol}(\mathbb{B}^{n-1})\mathrm{Vol}(\Gamma_1)}\bigg)^{\frac{1}{n-1}}+O(1).
\end{equation*}
\noindent In other words, an asymptotic of $\lambda^-(\mu_m)$ is exactly given by the Weyl's law restricted to the connected component boundary $\Gamma_1$. In the same way, one can prove :
\begin{equation*}
	\lambda^+(\mu_m)= 2\pi \bigg(\frac{m}{\mathrm{Vol}(\mathbb{B}^{n-1})\mathrm{Vol}(\Gamma_0)}\bigg)^{\frac{1}{n-1}}+O(1).
\end{equation*}
We recognize again the Weyl's law, restricted to the connected component boundary $\Gamma_0$.
\end{remark}

\medskip

\begin{remark}
The equalities proved in Lemma \ref{vp} highlight the link that exists between the Steklov spectrum and the spectrum of $-\Delta_{g_\mathbb{S}}$. Let us denote the eigenvalues of the Laplace-Beltrami operator on $\big(\mathbb{S},f(0)g_{\mathbb{S}}\big)$ by
	\begin{equation*}
		\mu_0^{(0)}\le \mu_1^{(0)}\le \mu_2^{(0)} \le ... \to +\infty
	\end{equation*}
\noindent with $\displaystyle \mu_m^{(0)}=\frac{\mu_m}{f(0)}$ for $m\in\mathbb{N}$. Lemma \ref{vp} implies in particular that there is a constant $C^{(0)}_f>0$ only depending on the conformal factor $f$ at $x=0$ such that
\begin{equation*}
	\big|\lambda^+(\mu_m)-\sqrt{\mu_m^{(0)}}\big|\le C_f^{(0)}
\end{equation*}
This can be related to results obtained in \cite{provenzano2019weyl} (Theorem 1.7 p.2) where it is proved that, for a bounded domain $\partial\Omega\subset \mathbb{R}^n$ with boundary of class $C^2$ which has only one boundary component, there is a bound $C_{\Omega}>0$ depending on $\Omega$ such that
\begin{equation*}
|\lambda_m-\sqrt{\mu_m}|\le C_{\Omega}\quad \forall m\in\mathbb{N}.
\end{equation*}
\noindent where $\lambda_m$ and $\mu_m$ are respectively the $m^{th}$ eigenvalue of the DN map and the the $m^{th}$ eigenvalue of the Laplace-Beltrami operator on the boundary. One can notice that, in our case, $C_f^{(0)}$ only depends on the metric on the boundary component $(\mathbb{S}^{n-1},f(0)g_\mathbb{S})$. This fact can be compared to a recent result (see \cite{colbois2019steklov}) where it is proved that the previous bound $C_\Omega$ can be chosen uniformly with respect to a class of manifolds $\mathcal{M}$ satisfying some geometrical conditions only in a neighborhood of the boundary (Theorem 3, p.3). 

\medskip

\noindent In the same way, Lemma \ref{vp} implies also the existence of $C^{(1)}_f>0$ only depending on the conformal factor $f$ at $x=1$ such that
\begin{equation*}
\big|\lambda^-(\mu_m)-\sqrt{\mu_m^{(1)}}\big|\le C^{(1)}_f
\end{equation*}
\noindent where $\displaystyle \mu_m^{(1)}=\frac{\mu_m}{f(1)}$ is the $m^{th}$ eigenvalue of the Laplace-Beltrami operator on $\big(\mathbb{S},f(1)g_{\mathbb{S}}\big)$.  
\end{remark}

\noindent Let us prove Lemma \ref{vp}.

\medskip

\begin{proof} We distinguish two cases :
	
	\medskip
	
				\begin{enumerate}[label=\alph*), align=left, leftmargin=*, noitemsep]
		
		\item[$\bullet$] Assume $f(0)\ne f(1)$ (for instance $f(0)<f(1)$).
	\end{enumerate}
	
	\medskip
	
		\noindent The characteristic polynomial $P(X)$ of $\Lambda_g^m(\lambda)$ is : 						
	
	\begin{equation*}
	P(X)=X^2-{\rm{Tr}}(\Lambda_g^m(\lambda))X+\det(\Lambda_g^m(\lambda)).
	\end{equation*}
	
	$\quad$
	
	\noindent To simplify the notation, we set :
	
	\medskip
	
\begin{center}
	$\displaystyle C_0=\frac{\ln(h)'(0)}{4\sqrt{f(0)}},\quad$ $\quad\displaystyle C_1=\frac{\ln(h)'(1)}{4\sqrt{f(1)}}$.
\end{center}

\medskip
	
	\noindent Thanks to Propositions \ref{estim} and \ref{Mat}, for $m$ large enough,  ${\rm{Tr}}(\Lambda_g^m(\lambda))$ and $\det(\Lambda_g^m(\lambda))$ satisfy:
	
	\begin{equation*}
	\left\{
	\begin{aligned}
	&{\rm{Tr}}(\Lambda_g^m(\lambda))=-\frac{M(\mu_m)}{\sqrt{f(0)}}-\frac{N(\mu_m)}{\sqrt{f(1)}}+C_0-C_1.	 \\
	&\det(\Lambda_g^m)=\bigg(-\frac{M(\mu_m)}{\sqrt{f(0)}}+C_0\bigg)\bigg(-\frac{N(\mu_m)}{\sqrt{f(1)}}-C_1\bigg)+O(\mu_m e^{-2\sqrt{\mu_m}}).
	\end{aligned}\right.
	\end{equation*}
	
	\medskip
	
	$\quad$
	
	\noindent The asymptotics of the discriminant $\delta$ of $P(X)$ depending on $M(\mu_m)$ and $N(\mu_m)$ can thus be written :
	\begin{equation*}
	\begin{aligned}
	\delta &= \bigg(-\frac{M(\mu_m)}{\sqrt{f(0)}}+C_0-\frac{N(\mu_m)}{\sqrt{f(1)}}-C_1\bigg)^2-4\bigg(-\frac{M(\mu_m)}{\sqrt{f(0)}}+C_0\bigg)\bigg(-\frac{N(\mu_m)}{\sqrt{f(1)}}-C_1\bigg)\\
	&\qquad\qquad\qquad\qquad\qquad\qquad\qquad\qquad\qquad\qquad\qquad\qquad\qquad\qquad\qquad\qquad+O(\mu_m e^{-2\sqrt{\mu_m}}).\\
	&=\bigg(-\frac{M(\mu_m)}{\sqrt{f(0)}}+C_0+\frac{N(\mu_m)}{\sqrt{f(1)}}+C_1\bigg)^2+O(\mu_m e^{-2\sqrt{\mu_m}}).
	\end{aligned}
	\end{equation*}
	
	\medskip
	
	\noindent Now, let us recall the result obtained by Simon in \cite{simon1999new} :
	
	\medskip
	
	\begin{theorem}
		\label{Simon}
		$M(z^2)$ has the following asymptotic expansion :
		\begin{equation*}
		\forall A\in\mathbb{N},\: \:-M(z^2)\underset{z \to \infty}{=}z+\sum_{j=0}^{A}\frac{\beta_j(0)}{z^{j+1}}+o\bigg(\frac{1}{z^{A+1}}\bigg)
		\end{equation*}
		\noindent where, for every $x\in[0,1]$, $\beta_j(x)$ is defined by : $\left\{\begin{aligned}
		&\beta_0(x)=\frac{1}{2}q(x)\\
		&\beta_{j+1}(x)=\frac{1}{2}\beta'_j(x)+\frac{1}{2}\sum_{l=0}^{j}\beta_l(x)\beta_{j-l}(x).
		\end{aligned}\right.$
	\end{theorem}
	
	\medskip
	
	\noindent Of course, by symmetry, one has immediately:
	
	\medskip
	
	\begin{corollary}
		\label{CorSimon}
		$N(z^2)$ has the following asymptotic expansion :
		\begin{equation*}
		\forall A\in\mathbb{N},\: \:-N(z^2)\underset{_{\substack{z \to \infty}}}{=}z+\sum_{j=0}^{A}\frac{\gamma_j(0)}{z^{j+1}}+o\bigg(\frac{1}{z^{A+1}}\bigg)
		\end{equation*}
		\noindent where, for all $x\in[0,1]$, $\gamma_j(x)$ is defined by : $\left\{\begin{aligned}
		&\gamma_0(x)=\frac{1}{2}q(1-x)\\
		&\gamma_{j+1}(x)=\frac{1}{2}\gamma'_j(x)+\frac{1}{2}\sum_{l=0}^{j}\gamma_l(x)\gamma_{j-l}(x).
		\end{aligned}\right.$
	\end{corollary}

\noindent
We deduce from Theorem \ref{Simon} and Corollary \ref{CorSimon}:
	
	\medskip
	
\begin{equation*}
-\frac{M(\mu_m)}{\sqrt{f(0)}}+\frac{N(\mu_m)}{\sqrt{f(1)}}=\underbrace{\bigg(\frac{1}{\sqrt{f(0)}}-\frac{1}{\sqrt{f(1)}}\bigg)}_{> 0}\sqrt{\mu_m}+O\bigg(\frac{1}{\sqrt{\mu_m}}\bigg).
\end{equation*}

\noindent
Thus, recalling that
\begin{equation*}
\begin{aligned}
\delta &=\bigg(-\frac{M(\mu_m)}{\sqrt{f(0)}}+C_0+\frac{N(\mu_m)}{\sqrt{f(1)}}+C_1\bigg)^2+O(\mu_m e^{-2\sqrt{\mu_m}}).
\end{aligned}
\end{equation*}

\noindent we obtain:

\begin{equation*}
\begin{aligned}
\sqrt{\delta}&=\bigg(\frac{N(\mu_m)}{\sqrt{f(1)}}-\frac{M(\mu_m)}{\sqrt{f(0)}}+C_0+C_1\bigg)\sqrt{1+O\bigg(e^{-2\sqrt{\mu_m}}\bigg)}\\
&=\frac{N(\mu_m)}{\sqrt{f(1)}}-\frac{M(\mu_m)}{\sqrt{f(0)}}+C_0+C_1+O\big(\sqrt{\mu_m}e^{-2\sqrt{\mu_m}}\big).
\end{aligned}
\end{equation*}

\medskip

\noindent Hence :
	
\begin{equation*}
\left\{
\begin{aligned}
&\lambda^-(\mu_m)=\frac{1}{2}\bigg[\bigg(-\frac{M(\mu_m)}{\sqrt{f(0)}}-\frac{N(\mu_m)}{\sqrt{f(1)}}+C_0-C_1\bigg)-\sqrt{\delta}\bigg] \\
&\lambda^+(\mu_m)=\frac{1}{2}\bigg[\bigg(-\frac{M(\mu_m)}{\sqrt{f(0)}}-\frac{N(\mu_m)}{\sqrt{f(1)}}+C_0-C_1\bigg)+\sqrt{\delta}\bigg],
\end{aligned}\right.
\end{equation*}
	
	\medskip
	
	\noindent and therefore, substituting $C_1$ and $C_2$ by their values and $M(\mu_m)$ and $N(\mu_m)$ by their asymptotics, we get :
	
	\begin{equation*}
	\left\{
	\begin{aligned}
	&\lambda^-(\mu_m)=\frac{\sqrt{\mu_m}}{\sqrt{f(1)}}-\frac{\ln(h)'(1)}{4\sqrt{f(1)}}+O\bigg(\frac{1}{\sqrt{\mu_m}}\bigg)	 \\
	&\lambda^+(\mu_m)=\frac{\sqrt{\mu_m}}{\sqrt{f(0)}}+\frac{\ln(h)'(0)}{4\sqrt{f(0)}}+O\bigg(\frac{1}{\sqrt{\mu_m}}\bigg).
	\end{aligned}\right.
	\end{equation*}
	
	\medskip
	
	\begin{enumerate}[label=\alph*), align=left, leftmargin=*, noitemsep]
	\item[$\bullet$] Assume now $f(0)=f(1)$. In this case, the restricted DN map
			\end{enumerate}
	
	\begin{equation*}
	\Lambda_g^m(\lambda)=\begin{pmatrix}
	-\frac{M(\mu_m)}{\sqrt{f(0)}}+\frac{\ln(h)'(0)}{4\sqrt{f(0)}}&-\frac{1}{\sqrt{f(0)}}\frac{1}{\Delta(\mu_m)}\\
	-\frac{1}{\sqrt{f(0)}}\frac{1}{\Delta(\mu_m)}&-\frac{N(\mu_m)}{\sqrt{f(1)}}-\frac{\ln(h)'(1)}{4\sqrt{f(1)}}
	\end{pmatrix}
	\end{equation*}

$\quad$

\noindent is a symmetric matrix and we can use the well-known result:
		
		\medskip
			
\begin{lemma}
	\label{qm}
Let $H$ be a Hilbert space, $A\in\mathcal{L}(H)$ be a selfadjoint operator. Let $\epsilon>0$. Assume there exists $\lambda_0\in\mathbb{R}$ and $u_0\in H$ a unit vector such that $\|(A-\lambda_0 Id)u_0\|\le\epsilon$. Then there exists an element $\lambda$ in the spectrum of $A$ such that $|\lambda-\lambda_0|\le\epsilon$.
\end{lemma}	
\noindent We apply this theorem with $A=\Lambda_g^m(\lambda)$, $\displaystyle\lambda_0=-\frac{M(\mu_m)}{\sqrt{f(0)}}+\frac{\ln(h)'(0)}{4\sqrt{f(0)}}$ and $U_0=\begin{pmatrix}
	1\\0
\end{pmatrix}$.

\medskip

\noindent We have :
\begin{equation*}
\displaystyle A-\lambda_0I_2=\begin{pmatrix}
0&&-\frac{1}{\sqrt{f(0)}}\frac{1}{\Delta(\mu_m)}\\
-\frac{1}{\sqrt{f(0)}}\frac{1}{\Delta(\mu_m)}&&-\frac{N(\mu_m)}{\sqrt{f(1)}}-\frac{\ln(h)'(1)}{4\sqrt{f(1)}}-\lambda_0
\end{pmatrix}
\end{equation*}

\medskip

\noindent Hence :
\begin{equation*}
(A-\lambda_0 I_2)U_0=-\frac{1}{\sqrt{f(0)}}\begin{pmatrix}
0\\ \frac{1}{\Delta(\mu_m)}
\end{pmatrix}=\displaystyle O(\sqrt{\mu_m}e^{-\sqrt{\mu_m}}).
\end{equation*}

\medskip

\noindent Lemma \ref{qm} gives $\lambda^+_m\in\sigma(\Lambda_g^m(\lambda))$ such that : \begin{equation*}
\begin{aligned}
\displaystyle \lambda^+_m&=-\frac{M(\mu_m)}{\sqrt{f(0)}}+\frac{\ln(h)'(0)}{4\sqrt{f(0)}}+O(\sqrt{\mu_m}e^{-\sqrt{\mu_m}})\\
&=\frac{\sqrt{\mu_m}}{\sqrt{f(0)}}+\frac{\ln(h)'(0)}{4\sqrt{f(0)}}+O\bigg(\frac{1}{\sqrt{\mu_m}}\bigg)
\end{aligned}
\end{equation*}

\medskip

\noindent from Theorem \ref{Simon}. The second eigenvalue $\lambda_m^2$ (which is distinct from $\lambda^1_m$ otherwhise $\Lambda_g^m(\lambda)$ would be a homothety) can be deduced from the first one :
\begin{equation*}
\begin{aligned}
\lambda_m^-&={\rm{Tr}}\big(\Lambda_g^m(\lambda)\big)-\lambda_m^+\\
&=-\frac{N(\mu_m)}{\sqrt{f(1)}}-\frac{\ln(h)'(1)}{4\sqrt{f(1)}}+O(\sqrt{\mu_m}e^{-\sqrt{\mu_m}})\\
&=\frac{\sqrt{\mu_m}}{\sqrt{f(1)}}-\frac{\ln(h)'(1)}{4\sqrt{f(1)}}+O\bigg(\frac{1}{\sqrt{\mu_m}}\bigg).
\end{aligned}
\end{equation*}

\end{proof}

\medskip

\begin{lemma}
	Under the hypothesis of Proposition \ref{trdet}, we have the following alternative :
	\begin{center}
	$\left\{
	\begin{aligned}
	&f(0)=\tilde{f}(0)	 \\
	&f(1)=\tilde{f}(1)
	\end{aligned}\right.\:\:$  or  $\:\:\left\{\begin{aligned}
	&f(0)=\tilde{f}(1)	 \\
	&f(1)=\tilde{f}(0).
	\end{aligned}\right.$
	\end{center}
\end{lemma}	

\medskip

\begin{proof}
We begin with:
\begin{equation*}
\frac{	{\rm{Tr}}(\Lambda_g^m(\lambda))}{\sqrt{\mu_m}}=	\frac{{\rm{Tr}}(\Lambda_{\tilde{g}}^m(\lambda))}{\sqrt{\mu_m}},\quad\forall m\in\mathbb{N}.
\end{equation*}
Thus, it follows from Lemma \ref{vp} that
\begin{equation*}
	\frac{1}{\sqrt{f(0)}}+\frac{1}{\sqrt{f(1)}}=\frac{1}{\sqrt{\tilde{f}(0)}}+\frac{1}{\sqrt{\tilde{f}(1)}}.
\end{equation*}

\noindent
In the same way, thanks to the relations:
\begin{equation*}
\frac{{\rm{det}}(\Lambda_g^m(\lambda))}{\mu_m}=	\frac{{\rm{det}}(\Lambda_{\tilde{g}}^m(\lambda))}{\mu_m}\quad\forall m\in\mathbb{N},
\end{equation*}
we get :
\begin{equation*}
	\frac{1}{\sqrt{f(0)f(1)}}=\frac{1}{\sqrt{\tilde{f}(0)\tilde{f}(1)}},
\end{equation*}
and the proof is complete.
\end{proof}

$\quad$

\noindent {\bf \underline{Case 1}} :  $f(0)=\tilde{f}(0)$ et $ f(1)=\tilde{f}(1)$.

$\quad$

\begin{lemma} Under the hypotheses of Proposition \ref{trdet}, we have:
\begin{equation}\label{Delta}
\Delta(z)=\tilde{\Delta}(z), \ \forall z\in\mathbb{C}.
\end{equation}
\end{lemma}

\begin{proof} We write the equality \rm{Tr}$(\Lambda_g^m(\lambda))$=Tr$(\Lambda^m_{\tilde{g}}(\lambda))$ as follows : for all $\mu_m\in\sigma(-\Delta_{g_K})$,

\begin{equation*}
\begin{aligned}
\bigg(-\frac{M(\mu_m)}{\sqrt{f(0)}}+\frac{\ln(h)'(0)}{4\sqrt{f(0)}}\bigg)+\bigg(-\frac{N(\mu_m)}{\sqrt{f(1)}}-\frac{\ln(h)'(1)}{4\sqrt{f(1)}}\bigg)=
\bigg(&-\frac{\tilde{M}(\mu_m)}{\sqrt{f(0)}}+\frac{\ln(\tilde{h})'(0)}{4\sqrt{f(0)}}\bigg)\\
&+\bigg(-\frac{\tilde{N}(\mu_m)}{\sqrt{f(1)}}-\frac{\ln(\tilde{h})'(1)}{4\sqrt{f(1)}}\bigg).
\end{aligned}
\end{equation*}

\vspace{0.2cm}\noindent
Thus, using Theorem \ref{Simon}, one gets when $m \rightarrow +\infty$,
\begin{equation}\label{sommeconst}
	\frac{\ln(h)'(0)}{4\sqrt{f(0)}}-\frac{\ln(h)'(1)}{4\sqrt{f(1)}}=\frac{\ln(\tilde{h})'(0)}{4\sqrt{f(0)}}-\frac{\ln(\tilde{h})'(1)}{4\sqrt{f(1)}}.
\end{equation}
It follows that
\begin{equation*}
	\frac{M(\mu_m)}{\sqrt{f(0)}}+\frac{N(\mu_m)}{\sqrt{f(1)}}=\frac{\tilde{M}(\mu_m)}{\sqrt{f(0)}}+\frac{\tilde{N}(\mu_m)}{\sqrt{f(1)}},
\end{equation*}
or equivalently
\begin{equation}\label{trace}
\tilde{\Delta}(\mu_m)\bigg(\frac{D(\mu_m)}{\sqrt{f(0)}}+\frac{E(\mu_m)}{\sqrt{f(1)}}\bigg)=\Delta(\mu_m)\bigg(\frac{\tilde{D}(\mu_m)}{\sqrt{f(0)}}+\frac{\tilde{E}(\mu_m)}{\sqrt{f(1)}}\bigg).
\end{equation}

\medskip

\noindent In the same way, using \rm{det}$(\Lambda_g^m(\lambda))$=\rm{det}$(\Lambda^m_{\tilde{g}}(\lambda))$, we have:
\begin{equation*}
\begin{aligned}
\bigg(\frac{M(\mu_m)}{\sqrt{f(0)}}-&\frac{\ln'(h)(0)}{4\sqrt{f(0)}}\bigg)\bigg(\frac{N(\mu_m)}{\sqrt{f(1)}}+\frac{\ln'(h)(1)}{4\sqrt{f(1)}}\bigg)-\frac{1}{\sqrt{f(0)f(1)}\Delta^2(\mu_m)}\\
&=\bigg(\frac{\tilde{M}(\mu_m)}{\sqrt{f(0)}}-\frac{\ln(\tilde{h})'(0)}{4\sqrt{f(0)}}\bigg)\times\bigg(\frac{\tilde{N}(\mu_m)}{\sqrt{f(1)}}+\frac{\ln'(\tilde{h})(1)}{4\sqrt{f(1)}}\bigg)-\frac{1}{\sqrt{f(0)f(1)}\tilde{\Delta}^2(\mu_m)}.
\end{aligned}
\end{equation*}

\medskip

\noindent Multiplying on both sides by $\Delta^2(\mu_m)\tilde{\Delta}^2(\mu_m)$, we obtain:

\begin{equation}\label{det}
\begin{aligned}
&\tilde{\Delta}^2(\mu_m)\bigg[\bigg(-\frac{D(\mu_m)}{\sqrt{f(0)}}-\Delta(\mu_m)\frac{\ln'(h)(0)}{4\sqrt{f(0)}}\bigg)\bigg(-\frac{E(\mu_m)}{\sqrt{f(1)}}
+\Delta(\mu_m)\frac{\ln'(h)(1)}{4\sqrt{f(1)}}\bigg)-\frac{1}{\sqrt{f(0)f(1)}}\bigg]\\
&=\Delta^2(\mu_m)\bigg[\bigg(-\frac{\tilde{D}(\mu_m)}{\sqrt{f(0)}}-\tilde{\Delta}(\mu_m)
\frac{\ln(\tilde{h})'(0)}{4\sqrt{f(0)}}\bigg)\bigg(-\frac{\tilde{E}(\mu_m)}{\sqrt{f(1)}}+\tilde{\Delta}(\mu_m)\frac{\ln'(\tilde{h})(1)}{4\sqrt{f(1)}}\bigg)-\frac{1}{\sqrt{f(0)f(1)}}\bigg]
\end{aligned}
\end{equation}

\vspace{0.2cm}
\noindent
Now, let us show that the equalities (\ref{trace}) and (\ref{det}) can be analytically extended with respect to $\mu_m$ in the half-right complex plane.
First, let us recall the definition of the so-called Nevanlinna class:

\medskip

\begin{definition}
	Set $\Pi^+=\{z\in\mathbb{C},\:\Re(z)>0\}$ the half-right plane of the complex plane. The Nevanlinna class $\mathcal{N}(\Pi^+)$ is the set of analytic functions $f$ on $\Pi^+$ such that
	\begin{equation*}
	\underset{0<r<1}{\sup}\int_{-\pi}^{\pi}\ln^+\bigg|f\bigg(\frac{1-re^{i\theta}}{1+re^{i\theta}}\bigg)\bigg|d\theta <\infty,
	\end{equation*}
	with:
	\begin{equation*}
	\ln^+(x)=\left\{\begin{aligned}
	&\ln x\:\:{\rm{if}}\:\: \ln(x)\ge 0\\
	&0\:\:{\rm{if}}\:\: \ln(x)< 0.
	\end{aligned}\right.
	\end{equation*}
\end{definition}

\medskip

\noindent We have the following result \cite{ramm1999inverse}:

\medskip

\begin{proposition}
	Let $h\in H(\Pi^+)$ an analytic function on $\Pi^+$, $A$ and $C$ two constants. Assume :
	\begin{equation*}
	|h(z)|\le Ce^{A\Re(z)},\quad \forall z\in\Pi^+.
	\end{equation*}
	Then $h\in\mathcal{N}(\Pi^+)$.
\end{proposition}

\noindent
Thus, thanks to the asymptotics of Proposition \ref{estim}, the holomorphic functions defined by $\delta:z\mapsto\Delta(z^2)$, $d:z\mapsto D(z^2)$ and $e:z\mapsto E(z^2)$
belong to $\mathcal{N}(\Pi^+)$. Let us recall now a useful uniqueness theorem involving functions in the Nevanlinna class (see \cite{ramm1999inverse} for instance):	

\begin{theorem}\label{nev}
	Let $h\in\mathcal{N}(\Pi^+)$ and $\mathcal{L}\subset\mathbb{R}_+$ be a countable set such that
	
	\medskip
	
	\begin{center}
		$\displaystyle\sum_{\mu\in\mathcal{L}}\frac{1}{\mu}=+\infty$.
	\end{center} Then :
	\begin{equation*}
	\big(h(\mu)=0,\:\forall \mu\in\mathcal{L}\big)\Rightarrow h\equiv 0 {\:\:\rm{on}}\:\Pi^+.
	\end{equation*}
\end{theorem}

$\quad$

\noindent Now, by the Weyl law,  (cf \cite{safarov1997asymptotic}), we get:

\begin{equation*}
	\mu_m\underset{m\to +\infty}{=}c(K)m^{\frac{2}{n-1}}+O(1)
\end{equation*}
where $\displaystyle c(K)=\frac{(2\pi)^2}{\big(\omega_1\mathrm{vol}(K)\big)^{\frac{2}{n-1}}}$ and $\omega_1$ is the volume of the unit ball in $\mathbb{R}^{n-1}$. Thus, for a fixed $T\in\mathbb{N}$, we have :
\begin{equation*}
	\frac{\mu_{(mT)^{n-1}}}{c(K)T^2}=m^2+\frac{O(1)}{c(K)T^2}.
\end{equation*}

\noindent
As a consequence, for $m$ and $T$ large enough, the real numbers $\mu_{(mT)^{n-1}}$ are always distinct. Now, we set :
\begin{equation*}
	\mathcal{L}=\{\sqrt{\mu_{(mT)^{n-1}}},\: m\in\mathbb{N}\}.
\end{equation*}
Using $\displaystyle\sqrt{\mu_{(mT)^{n-1}}}\underset{m\to+\infty}{\sim} \sqrt{c(K)}Tm$, one has:

\begin{equation*}
	\sum_{\mu\in\mathcal{L}}\frac{1}{\mu}=+\infty.
\end{equation*}
Thus, thanks to Theorem \ref{nev}, the relations (\ref{trace}) and (\ref{det}) are still true if one replaces $\mu_m$ by $z^2 \in \mathbb{C}$, then by $z$.

$\quad$

\noindent
Now, let us prove that  $\Delta(z)=\tilde{\Delta}(z)$ for any $z\in\mathbb{C}$.
We recall that these functions are entire of order $\displaystyle \frac{1}{2}$ and their roots are simple. So, using Hadamard's factorization theorem (see Proposition \ref{Prodinf}), we deduce that these functions
are entirely described by their roots (up to a multiplicative constant). Consequently, in order to prove that $\Delta=\tilde{\Delta}$, it is enough to show that their roots are the same.

\par
 Set $\mathcal{P}=\{\alpha_j\in\mathbb{C},\: \Delta(\alpha_j)=0\}$. Let $\alpha_k$ be in $\mathcal{P}$ and let us show that $\tilde{\Delta}(\alpha_k)=0$. By definition, $-\alpha_k$ is an eigenvalue of the Sturm Liouville operator $\displaystyle H = -\frac{d^2}{dx^2}+q$ with Dirichlet
boundary conditions. Thus, from Proposition \ref{racDelt}, $\alpha_k$ is real and since the potential $q$ is real, the quantities $D(\alpha_k)$ and $E(\alpha_k)$ are also real. Using (\ref{trace})
and (\ref{det}) with $\mu_m$ replaced with $\alpha_k$, we obtain the following system:

\begin{equation*}
\left\{\begin{aligned}
&\bigg(\frac{D(\alpha_k)}{\sqrt{f(0)}}+\frac{E(\alpha_k)}{\sqrt{f(1)}}\bigg)\tilde{\Delta}(\alpha_k)=0\\
&\bigg(\frac{D(\alpha_k)E(\alpha_k)-1}{\sqrt{f(0)f(1)}}\bigg)\tilde{\Delta}(\alpha_k)^2=0.
\end{aligned}\right.
\end{equation*}
To finish the proof, we distinguish two cases:

\medskip

\begin{enumerate}[label=\alph*), align=left, leftmargin=*, noitemsep]
	\item[$\bullet$] If $\displaystyle \frac{D(\alpha_k)}{\sqrt{f(0)}} +\frac{E(\alpha_k)}{\sqrt{f(1)}}\ne 0$ then $\tilde{\Delta}(\alpha_k)=0$.
	\item[$\bullet$] Otherwise $\displaystyle\frac{D(\alpha_k)}{\sqrt{f(0)}}=-\frac{E(\alpha_k)}{\sqrt{f(1)}}$. Then, substituting in the second equality, we have:
	\begin{equation*}
	\bigg(\frac{D(\alpha_k)^2}{f(0)}+\frac{1}{\sqrt{f(0)f(1)}}\bigg)\tilde{\Delta}(\alpha_k)^2=0
	\end{equation*}
\end{enumerate}
Since $D(\alpha_k)\in\mathbb{R}$, we conclude that $\tilde{\Delta}(\alpha_k)=0$. By a standard symmetric argument, it follows $\Delta$ and $\tilde{\Delta}$ have the same zeros. So, there exists $C$ such that
$\Delta=C\tilde{\Delta}$. But,  as $\Delta(z)$ and $\tilde{\Delta}(z)$ have the same asymptotics :
\begin{equation*}
	\Delta(z),\tilde{\Delta}(z)\sim\frac{\sinh(\sqrt{z})}{\sqrt{z}},
\end{equation*}
we deduce that $C=1$ and the proof is complete.

\end{proof}

\medskip

\begin{remark}
	It is interesting to note that $\Delta(z)=\tilde{\Delta}(z)$ for all $z\in\mathbb{C}$ implies that (in fact is equivalent to) $q$ and $\tilde{q}$ are isospectral.
\end{remark}

$\quad$

\noindent As a consequence, we can simplify (\ref{det}) (with $\mu_m$ replaced by $z$), and we get on $\mathbb{C}$ :

\begin{equation*}
\left\{\begin{aligned}
&\bigg(\frac{M(z)}{\sqrt{f(0)}}-\frac{\ln'(h)(0)}{4\sqrt{f(0)}}\bigg)+\bigg(\frac{N(z)}{\sqrt{f(1)}}+
\frac{\ln'(h)(1)}{4\sqrt{f(1)}}\bigg)=\bigg(\frac{\tilde{M}(z)}{\sqrt{f(0)}}-\frac{\ln(\tilde{h})'(0)}{4\sqrt{f(0)}}\bigg)\\
&\qquad\qquad\qquad\qquad\qquad\qquad\qquad\qquad\qquad\qquad\qquad\qquad\qquad+
\bigg(\frac{\tilde{N}(z)}{\sqrt{f(1)}}+\frac{\ln'(\tilde{h})(1)}{4\sqrt{f(1)}}\bigg)\\
&\bigg(\frac{M(z)}{\sqrt{f(0)}}-\frac{\ln'(h)(0)}{4\sqrt{f(0)}}\bigg)\bigg(\frac{N(z)}{\sqrt{f(1)}}+\frac{\ln'(h)(1)}{4\sqrt{f(1)}}\bigg)=\bigg(\frac{\tilde{M}(z)}{\sqrt{f(0)}}-
\frac{\ln(\tilde{h})'(0)}{4\sqrt{f(0)}}\bigg)\\
&\qquad\qquad\qquad\qquad\qquad\qquad\qquad\qquad\qquad\qquad\qquad\qquad\qquad\times\bigg(\frac{\tilde{N}(z)}{\sqrt{f(1)}}+\frac{\ln'(\tilde{h})(1)}{4\sqrt{f(1)}}\bigg).
\end{aligned}\right.
\end{equation*}
Thus, for each $z\in\mathbb{C}$, we have the following alternative:

\begin{equation*}
\label{}
	\left\{\begin{aligned}
	&\frac{M(z)}{\sqrt{f(0)}}-\frac{\ln'(h)(0)}{4\sqrt{f(0)}}=\frac{\tilde{M}(z)}{\sqrt{f(0)}}-\frac{\ln(\tilde{h})'(0)}{4\sqrt{f(0)}}\\
	&\frac{N(z)}{\sqrt{f(1)}}+\frac{\ln'(h)(1)}{4\sqrt{f(1)}}=\frac{\tilde{N}(z)}{\sqrt{f(1)}}+\frac{\ln'(\tilde{h})(1)}{4\sqrt{f(1)}},
	\end{aligned}\right.
\end{equation*}
or \begin{equation*}
	\left\{\begin{aligned}
	&\frac{M(z)}{\sqrt{f(0)}}-\frac{\ln'(h)(0)}{4\sqrt{f(0)}}=\frac{\tilde{N}(z)}{\sqrt{f(1)}}+\frac{\ln'(\tilde{h})(1)}{4\sqrt{f(1)}}\\
	&\frac{N(z)}{\sqrt{f(1)}}+\frac{\ln'(h)(1)}{4\sqrt{f(1)}}=\frac{\tilde{M}(z)}{\sqrt{f(0)}}-\frac{\ln(\tilde{h})'(0)}{4\sqrt{f(0)}}.
	\end{aligned}\right.
\end{equation*}
$\quad$

\medskip

\noindent $\bullet\:$ As a first step, assume that the first system is satisfied for all $z$ in $\mathbb{C}$. Then, using the asymptotics of $M(z)$, (resp.  $\tilde{M}(z))$, we get:

\begin{equation*}
\displaystyle \frac{\ln'(h)(0)}{4\sqrt{f(0)}}=\frac{\ln(\tilde{h})'(0)}{4\sqrt{f(0)}}
\end{equation*}
and therefore :
\begin{equation*}
M(z)=\tilde{M}(z) \ {\rm{for \ all}}\:\:z \in \mathbb{C}\backslash \mathbb{R}.
\end{equation*}

\medskip

\noindent Thus, thanks to the Borg-Marchenko's theorem (see \cite{simon1999new} for instance), we deduce:

\begin{equation}
\label{egq}
q=\tilde{q}\quad{\rm{sur}}\:\:[0,1].
\end{equation}
Recall that :
\begin{equation*}
q=\frac{(h^{1/4})''}{h^{1/4}}-\lambda f
\end{equation*}
where we have set $h=f^{n-2}$ (respectively $\tilde{q}$ and $\tilde{h}=\tilde{f}^{n-2}$). In particular, in the $2$ dimensional case, we get immediateley $ f=\tilde{f}$ since $\lambda\ne 0$.
In dimension $n \geq 3$, we see that $f$ and $\tilde{f}$ same verify the same ODE with $f(0)=\tilde{f}(0)$. Moreover, the relation

\begin{equation*}
	\displaystyle \frac{\ln'(h)(0)}{4\sqrt{f(0)}}=\frac{\ln(\tilde{h})'(0)}{4\sqrt{\tilde{f}(0)}},
\end{equation*}
implies :
\begin{equation*}
	f'(0)=\tilde{f}'(0).
\end{equation*}
Thus, the Cauchy-Lipschitz's theorem says that $f=\tilde{f}.$.

$\quad$

\noindent $\bullet\:$ Secondly, assume there exists $z_0\in\mathbb{C}$ satisfying :

\begin{equation*}
\displaystyle \frac{M(z_0)}{\sqrt{f(0)}}-\frac{\ln'(h)(0)}{4\sqrt{f(0)}}\ne \frac{\tilde{M}(z_0)}{\sqrt{f(0)}}-\frac{\ln(\tilde{h})'(0)}{4\sqrt{\tilde{f}(0)}}.
\end{equation*}

\medskip

\noindent By a standard continuity argument, there is a ball $B$ of center $z_0$ such that :

\begin{equation*}
\displaystyle \frac{M(z)}{\sqrt{f(0)}}-\frac{\ln'(h)(0)}{4\sqrt{f(0)}}\ne \frac{\tilde{M}(z)}{\sqrt{f(0)}}-\frac{\ln(\tilde{h})'(0)}{4\sqrt{\tilde{f}(0)}},\quad\forall z\in B.
\end{equation*} Then, necessarily, we have:
\begin{equation}
\label{eg}
\displaystyle \frac{M(z)}{\sqrt{f(0)}}-\frac{\ln'(h)(0)}{4\sqrt{f(0)}}=\frac{\tilde{N}(z)}{\sqrt{f(1)}}+\frac{\ln(\tilde{h})'(1)}{4\sqrt{\tilde{f}(1)}},\quad\forall z\in B.
\end{equation}
Then, using the analytic continuation principle, the previous equality is true for every $z\in \mathbb{C}\textbackslash\mathcal{P}$, where $\mathcal{P}$ is the set of roots of $\Delta(z)$. Thanks to the asymptotics of $M(z)$ and $\tilde{N}(z)$, one gets :
\begin{equation*}
 	f(0)=f(1)\:\:\big(=\tilde{f}(1)\big)\quad\mathrm{and}\quad \frac{\ln'(h)(0)}{4\sqrt{f(0)}}=-\frac{\ln(\tilde{h})'(1)}{4\sqrt{\tilde{f}(1)}}.
\end{equation*}
Hence, simplifying in (\ref{eg}), we obtain
 \begin{equation*}
 	M(z)=\tilde{N}(z) \ {\rm{for \ all}}\:\: z \in \mathbb{C}\backslash \mathbb{R}.
 \end{equation*}
By symmetry, $\tilde{N}$ has the same role as $\tilde{M}$ for the potential $x\mapsto \tilde{q}(1-x)$. Now, it follows, from the Borg-Marchenko's theorem that:
\begin{equation*}
q(x)=\tilde{q}(1-x)\quad\forall x\in\:[0,1],
\end{equation*}
and as previously, one gets:
\begin{equation*}
f=\tilde{f}\circ\eta.
\end{equation*}

$\quad$

\noindent {\bf \underline{Case 2} : ${\bf f(0)=\tilde{f}(1)}$ and ${\bf f(1)=\tilde{f}(0)}$}.

$\quad$

\noindent
The proof is identical interchanging the roles of $M$ and $N$.

$\quad$

\section{Uniqueness results on the trace and the determinant}

In this section, we assume that $(K,g_K)=(\mathbb{S}^{n-1},g_{\mathbb{S}})$, and our main result is the following:

\begin{proposition}
Assume that $\sigma(\Lambda_{g}(\lambda))=\sigma(\Lambda_{\tilde{g}}(\lambda))$. Then, there is $m_0\in\mathbb{N}$ such that :
\begin{center}
	$\forall m\in\mathbb{N}$, $m\ge m_0\Rightarrow$ $\det(\Lambda_g^m(\lambda))=\det(\Lambda_{\tilde{g}}^m(\lambda))$ and Tr$(\Lambda_g^m(\lambda))=$Tr$(\Lambda_{\tilde{g}}^m(\lambda))$.
\end{center}
\end{proposition}

\vspace{0.2cm}
\noindent
Before giving the proof of this proposition, let us begin by the following lemma:

\begin{lemma}
	Under the hypothesis $\sigma(\Lambda_{g}(\lambda))=\sigma(\Lambda_{\tilde{g}}(\lambda))$, we have the alternative :
	\begin{center}
		$\left\{
	\begin{aligned}
	&f(0)=\tilde{f}(0)	 \\
	&f(1)=\tilde{f}(1)
	\end{aligned}\right.\:\:$  or  $\:\:\left\{\begin{aligned}
	&f(0)=\tilde{f}(1)	 \\
	&f(1)=\tilde{f}(0).
	\end{aligned}\right.$
	\end{center}

\end{lemma}	

\begin{proof} First, let us define  the set:
	\begin{equation*}
	\Sigma\big(\Lambda_g(\lambda)\big)=\{\lambda^\pm(\kappa_m),\quad m\in\mathbb{N}\}
\end{equation*}
where $\kappa_m$ is the $m-$th eigenvalue of the usual Laplace-Beltrami operator on the sphere $\mathbb{S}^{n-1}$ and counted without multiplicity. Thanks to our hypothesis, one has obviously	
\begin{equation*}
\Sigma\big(\Lambda_g(\lambda)\big)=\Sigma\big(\Lambda_{\tilde{g}}(\lambda)\big).
\end{equation*}
When $K=\mathbb{S}^{n-1}$, we have an explicit formula for $\kappa_m$ (see for instance \cite{shubin1987pseudodifferential}) :
	\begin{equation*}
	\kappa_m=m(m+n-2),\quad \forall m\in\mathbb{N}.
	\end{equation*}
	
	$\quad$
	
	\noindent The proof involves two steps.
	
	\medskip

		\noindent {\bf Step 1}. First we have :
		\begin{equation}
		\label{coef}f(0)^{\frac{n-1}{2}}+f(1)^{\frac{n-1}{2}}=\tilde{f}(0)^{\frac{n-1}{2}}+\tilde{f}(1)^{\frac{n-1}{2}}.
		\end{equation} 
	
	$\quad$
	
	\noindent \noindent Indeed, it is known that the Steklov spectrum determines the volume of the boundary of $M$ : this is an immediate consequence of the Weyl's law (\ref{WeylSteklov}) for Steklov eigenvalues (see \cite{girouard2017spectral}).
	\noindent We have $\partial M=\Gamma_0\cup\Gamma_1$ where, for $i\in\{0,1\}$, $\Gamma_i$ is the sphere $\mathbb{S}^{n-1}$ equipped with the metric $\gamma_i=f(i)g_{\mathbb{S}}$. Hence :
		\begin{equation*}
		\begin{aligned}
			\mathrm{Vol}\big(\partial M\big)=\mathrm{Vol}\big(\partial \tilde{M}\big)&\Leftrightarrow \int_{\Gamma_0}\,d\mathrm{Vol_{\gamma_0}} + \int_{\Gamma_1}\,d\mathrm{Vol_{\gamma_1}} = \int_{\Gamma_0}\,d\mathrm{Vol_{{\tilde\gamma}_0}}+\int_{\Gamma_1}\,d\mathrm{Vol_{{\tilde\gamma}_1}}\\
			&\Leftrightarrow \bigg(f(0)^{\frac{n-1}{2}}+f(1)^{\frac{n-1}{2}}\bigg)\mathrm{Vol}(\mathbb{S}^{n-1})=\bigg(\tilde{f}(0)^{\frac{n-1}{2}}+\tilde{f}(1)^{\frac{n-1}{2}}\bigg)\mathrm{Vol}(\mathbb{S}^{n-1}),
		\end{aligned}
	\end{equation*}
	
	\noindent and this proves the claim.
	
	$\quad$
	
 \noindent {\bf Step 2.} We show that : $f(0)\in\{\tilde{f}(0),\tilde{f}(1)\}$.

\medskip

\noindent Assume this statement is false. Without loss of generality, assume that $\displaystyle f(0)<\min\{\tilde{f(0)},\tilde{f}(1)\}$. Then the equality (\ref{coef}) implies : $f(1)>\max\{\tilde{f}(0),\tilde{f}(1)\}$.

$\quad$

\noindent Our strategy is the following : we prove that one of the elements of $\Sigma\big(\Lambda_g(\lambda)\big)$ is not in $\Sigma\big(\Lambda_{\tilde{g}}(\lambda)\big)$ and this shall give a contradiction.

$\quad$

\noindent Let $\varepsilon$ and $A$ two positive numbers. Set :

\begin{equation*}
\alpha_A=\inf\{|\lambda^+(\kappa_n)-\lambda^-(\kappa_m)|\:|\: m,n\ge A\}.
\end{equation*}

$\quad$

\noindent  We claim that $\displaystyle \alpha_A\le \frac{1}{2\sqrt{f(1)}}+\varepsilon$ for a sufficiently large $A$. 

\medskip

$\quad$

\noindent Indeed, let $n\in\mathbb{N}$ and set \begin{center}
	$\displaystyle m=\max\bigg\{j\in\mathbb{N}\:\:\big|\:\: \lambda^-(\kappa_j)<\lambda^+(\kappa_n)-\frac{1}{2\sqrt{f(1)}}-\frac{\varepsilon}{2}   \bigg\}$
\end{center}

\noindent By definition of $m$, $\displaystyle \lambda^-(\kappa_{m+1})\ge \lambda^+(\kappa_n)-\frac{1}{2\sqrt{f(1)}}-\varepsilon$. But :
\begin{equation*}
\begin{aligned}
\lambda^-(\kappa_{m+1})&=\lambda^-(\kappa_{m})+\frac{1}{\sqrt{f(1)}}+O\bigg(\frac{1}{m}\bigg)\\
&< \lambda^+(\kappa_{n})-\frac{1}{2\sqrt{f(1)}}-\frac{\varepsilon}{2}+\frac{1}{\sqrt{f(1)}}+O\bigg(\frac{1}{m}\bigg)\\
&<\lambda^+(\kappa_{n})+\frac{1}{2\sqrt{f(1)}}+\frac{\varepsilon}{2}+O\bigg(\frac{1}{m}\bigg)
\end{aligned}
\end{equation*}

\noindent Hence 
\begin{center}
	$\displaystyle \alpha_A\le |\lambda^+(\kappa_{n})-\lambda^-(\kappa_{m+1})|\le \frac{1}{2\sqrt{f(1)}}+\varepsilon,\quad$ for $m,n\ge A$ sufficiently large.
\end{center}

$\quad$

\noindent From the above, we can now choose $(m,n)\in\mathbb{N}^2$ such that :

\begin{equation}
\label{in1}
-\frac{1}{2\sqrt{f(1)}}-\varepsilon\le \lambda^+(\kappa_{n})-\lambda^-(\kappa_{m})\le \frac{1}{2\sqrt{f(1)}}+\varepsilon.
\end{equation}

\noindent The equality of the sets $\Sigma\big(\Lambda_g(\lambda)\big)$ and $\Sigma\big(\Lambda_{\tilde{g}}(\lambda)\big)$ gives two elements $\tilde{\lambda}_1$ and $\tilde{\lambda}_2$ of $\Sigma\big(\Lambda_{\tilde{g}}(\lambda)\big)$ such that $\displaystyle \{\lambda^-(\kappa_{m}),\lambda^+(\kappa_{n})\}=\{\tilde{\lambda}_1,\tilde{\lambda}_2\}$. Remark that if $\displaystyle \lambda^-(\kappa_{m})\ne\lambda^+(\kappa_{n})$ then the elements $\tilde{\lambda}_1$ and $\tilde{\lambda}_2$ can not both belong to the same sequence $(\tilde{\lambda}^-(\kappa_\ell))$ or $(\tilde{\lambda}^+(\kappa_p))$ because, for indexes $m$ and $n$ large enough, one would get :
\begin{equation*}
|\tilde{\lambda}_1-\tilde{\lambda}_2|\ge \min\bigg\{\frac{1}{\sqrt{\tilde{f}(0)}},\frac{1}{\sqrt{\tilde{f}(1)}}\bigg\}-\varepsilon
\end{equation*}
and, by choosing $\varepsilon$ small enough, this would contradict (\ref{in1}). Consequently, we have the existence of $(\ell,p)\in\mathbb{N}^2$ such that
\begin{equation*}
	\{\lambda^-(\kappa_{m}),\lambda^+(\kappa_{n})\}=\{\tilde{\lambda}^-(\kappa_{\ell}),\tilde{\lambda}^+(\kappa_{p})\}.
\end{equation*}

\noindent Let us assume first that
\begin{equation}
\label{eqset}
\left\{
\begin{aligned}
& \lambda^-(\kappa_m)= \tilde{\lambda}^-(\kappa_\ell) \\
& \lambda^+(\kappa_n)= \tilde{\lambda}^+(\kappa_p).  \end{aligned} \right.
\end{equation}

\begin{enumerate}
	\item [$*$] Case 1 : $\lambda^-(\kappa_m)\le \lambda^+(\kappa_n)$. Then $\lambda^-(\kappa_{m+1})$ is not equal to any element of $\Sigma(\Lambda_{\tilde{g}}(\lambda))$. Indeed we have on one hand :
	
	\begin{equation*}
	\begin{aligned}
	\lambda^-(\kappa_{m+1})&=\lambda^-(\kappa_{m})+\frac{1}{\sqrt{f(1)}}+O\bigg(\frac{1}{m}\bigg)\\
	&=\tilde{\lambda}^-(\kappa_{\ell})+\frac{1}{\sqrt{f(1)}}+O\bigg(\frac{1}{m}\bigg)\\
	&=\tilde{\lambda}^-(\kappa_{\ell+1})+\bigg[\underbrace{\frac{1}{\sqrt{f(1)}}-\frac{1}{\sqrt{\tilde{f}(1)}}}_{<0}\bigg]+O\bigg(\frac{1}{m}\bigg)+O\bigg(\frac{1}{\ell}\bigg)\\
	\end{aligned}
	\end{equation*}
	
	\noindent Hence, for $m$ and $\ell$ large enough : $\tilde{\lambda}^-(\kappa_{\ell})<\lambda^-(\kappa_{m+1})<\tilde{\lambda}^-(\kappa_{\ell+1})$. On the other hand :
	
	\begin{equation*}
	\begin{aligned}
	\lambda^-(\kappa_{m+1})&= \lambda^-(\kappa_{m})+\frac{1}{\sqrt{f(1)}}+O\bigg(\frac{1}{m}\bigg)\\
	&\le \lambda^+(\kappa_{n})+\frac{1}{\sqrt{f(1)}}+O\bigg(\frac{1}{m}\bigg)\\
	&= \tilde{\lambda}^+(\kappa_{p})+\frac{1}{\sqrt{f(1)}}+O\bigg(\frac{1}{m}\bigg)\\
	&=\tilde{\lambda}^+(\kappa_{p+1})+\bigg[\underbrace{\frac{1}{\sqrt{f(1)}}-\frac{1}{\sqrt{\tilde{f}(0)}}}_{<0}\bigg]+O\bigg(\frac{1}{m}\bigg)+O\bigg(\frac{1}{p}\bigg).
	\end{aligned}
	\end{equation*}
	\noindent Moreover, from (\ref{in1}) and (\ref{eqset}) : \begin{equation*}
	\begin{aligned}
	\displaystyle \lambda^-(\kappa_{m+1})&=\lambda^-(\kappa_{m})+\frac{1}{\sqrt{f(1)}}+O\bigg(\frac{1}{m}\bigg)\\&=\lambda^-(\kappa_{m})+\frac{1}{2\sqrt{f(1)}}+\varepsilon+O\bigg(\frac{1}{m}\bigg)+\frac{1}{2\sqrt{f(1)}}-\varepsilon\\&\ge \tilde{\lambda}^+(\kappa_p)+\frac{1}{2\sqrt{f(1)}}-2\varepsilon
	\end{aligned}
	\end{equation*}
	
	$\quad$

	\noindent As a result, for $m$ and $p$ large enough : $\displaystyle \tilde{\lambda}^+(\kappa_{p})<\lambda^-(\kappa_{m+1})<\tilde{\lambda}^+(\kappa_{p+1})$
	
	\medskip
	
	\noindent The sequences $(\tilde{\lambda}^-(\kappa_m))$ and $(\tilde{\lambda}^+(\kappa_m))$ being strictly increasing (at least for $m$ large enough), none of the elements of $\Sigma(\Lambda_{\tilde{g}}(\lambda))$ can be equal to $\lambda^-(\kappa_{m+1})$. This refutes $\Sigma(\Lambda_{\tilde{g}}(\lambda))=\Sigma(\Lambda_{g}(\lambda))$.
	
	$\quad$
	
	\item[$*$] Case 2 : $\lambda^-(\kappa_m)> \lambda^+(\kappa_n)$. With similar arguments, one can prove :
	\begin{equation*}
	\max\big(\tilde{\lambda}^-(\kappa_{\ell-1}), \tilde{\lambda}^+(\kappa_{p-1})\big)< \lambda^-(\kappa_{m-1}) < \min\big(\tilde{\lambda}^-(\kappa_{\ell}), \tilde{\lambda}^+(\kappa_{p})\big).
	\end{equation*}
	
\end{enumerate}

\noindent Now, if we assume that
\begin{equation*}
\left\{
\begin{aligned}
& \lambda^-(\kappa_m)= \tilde{\lambda}^+(\kappa_p) \\
& \lambda^+(\kappa_n)= \tilde{\lambda}^-(\kappa_\ell),  \end{aligned} \right.
\end{equation*}

\noindent one can also prove, by interchanging the roles of $\tilde{f}(0)$ and $\tilde{f}(1)$, that $\lambda^-(\kappa_{m+1})$ or $\lambda^-(\kappa_{m-1})$ does not belong to the set $\Sigma(\Lambda_{\tilde{g}}(\lambda))$.

$\quad$ 

\noindent Thus $f(0)\in\{\tilde{f}(0),\tilde{f}(1)\}$. Associated to the equality (\ref{coef}), this gives the wanted conclusion.

\end{proof}

\vspace{1cm}
\medskip

\noindent From now on, we assume that $f(0)=\tilde{f}(0)$ and $f(1)=\tilde{f}(1)$, since the case 
 $f(0)=\tilde{f}(1)$ and $f(1)=\tilde{f}(0)$ is obtained by substituting the roles of $\tilde{\lambda}^-(\kappa_m)$ and $\tilde{\lambda}^+(\kappa_m)$. First, let us begin by a simple case:

\subsection{The case ${\bf f(0)= f(1)}$}

\medskip

\noindent Without loss of generality, we can assume that $f(0)=f(1)=1$. Thanks to Lemma \ref{vp}, $\Lambda_g^m(\lambda)$ has two eigenvalues $\lambda^-(\mu_m)$ and $\lambda^+(\mu_m)$ whose asymptotics are given by :
\begin{equation*}
\left\{
\begin{aligned}
&\lambda^-(\mu_m)=\sqrt{\mu_m}-\frac{\ln(h)'(1)}{4\sqrt{f(1)}}+O\bigg(\frac{1}{\sqrt{\mu_m}}\bigg)	 \\
&\lambda^+(\mu_m)=\sqrt{\mu_m}+\frac{\ln(h)'(0)}{4\sqrt{f(0)}}+O\bigg(\frac{1}{\sqrt{\mu_m}}\bigg).
\end{aligned}\right.
\end{equation*}
We recall that : \begin{equation*}
	\sqrt{\kappa_m}=m+\frac{n-2}{2}+O\bigg(\frac{1}{m}\bigg),
\end{equation*}
and for $m\in\mathbb{N}$, we set $\displaystyle V_m=\bigg\{\lambda^-(\kappa_m),\:\lambda^+(\kappa_m)\bigg\}$ and $\displaystyle \tilde{V}_m:=\bigg\{\tilde{\lambda}^-(\kappa_m),\:\tilde{\lambda}^+(\kappa_m)\bigg\}$.

\vspace{1.5cm}
\noindent
$\bullet$ {\bf{The two dimensional case}}

\vspace{0.5cm}
\noindent 
In this case, the eigenvalues $\lambda^\pm(\kappa_m)$ have the following asymptotics :

\begin{equation*}
\left\{
\begin{aligned}
&\lambda^-(\kappa_m)=m+O\bigg(\frac{1}{m}\bigg)	 \\
&\lambda^+(\kappa_m)=m+O\bigg(\frac{1}{m}\bigg).
\end{aligned}\right.
\end{equation*}

\noindent For $m$ large enough, the sets $V_m$ and $\tilde{V}_m$ are both included in the interval $\displaystyle \big[m-\frac{1}{4},m+\frac{1}{4}\big]$. In particular $V_m\cap \tilde{V}_{m'}=\emptyset\:$ if $\:\:m\ne m'$. The equality $\Sigma\big(\Lambda_g(\lambda)\big)=\Sigma\big(\Lambda_{\tilde{g}}(\lambda)\big)$ leads to the equalities $\displaystyle V_m=\tilde{V}_m$ if $m$ is greater than some index $m_0$. Consequently, there is $m_0$ such that, for $m\ge m_0$ :
\begin{equation*}
\left\{\begin{aligned}
	&\lambda^-(\kappa_m)+\lambda^+(\kappa_m)=\tilde{\lambda}^-(\kappa_m)+\tilde{\lambda}^+(\kappa_m) \\
	&\lambda^-(\kappa_m)\lambda^+(\kappa_m)=\tilde{\lambda}^-(\kappa_m)\tilde{\lambda}^+(\kappa_m).
	\end{aligned}\right.
\end{equation*}

\medskip

\noindent Of course, the previous equalities are still true when $\kappa_m$ is replaced by $\mu_m$. Thus, we have proved:

\medskip

\begin{center}
	$\forall m\in\mathbb{N},\:\:m\ge m_0\Rightarrow$ Tr$(\Lambda^m_g)=\: $Tr$(\Lambda^m_{\tilde{g}})$ and $\det(\Lambda^m_g)=\det(\Lambda^m_{\tilde{g}})$.
\end{center}

\vspace{1cm}
\noindent
$\bullet$ {\bf{The $n \geq 3$ dimensional case}}

\medskip

\noindent In this case, we use the following asymptotics of the eigenvalues $\lambda^\pm(\kappa_m)$:

\begin{equation*}
\left\{
\begin{aligned}
&\lambda^-(\kappa_m)=m+\frac{n-2}{2}-\frac{h'(1)}{4}+O\bigg(\frac{1}{m}\bigg)	 \\
&\lambda^+(\kappa_m)=m+\frac{n-2}{2}+\frac{h'(0)}{4}+O\bigg(\frac{1}{m}\bigg).
\end{aligned}\right.
\end{equation*}

\medskip

\noindent 
Contrary to the two dimensional case, we can not conclude that the sets $V_m$ and $\tilde{V}_m$ are equal. This is due to the presence of the constants $h'(0)$ and $h'(1)$ in these asymptotics. We have the following Proposition:

\begin{proposition}
	\label{eq_trans_1}
There is $\mathcal{L}\subset\mathbb{N}$ such that $\displaystyle \sum_{m\in\mathcal{L}}\frac{1}{m}=+\infty$ satisfying :
	\begin{equation*}
 \forall m\in\mathcal{L},\:\:\left\{
	\begin{aligned}
	& \lambda^-(\kappa_{m})+\lambda^+(\kappa_{m})=\tilde{\lambda}^-(\kappa_{m})+\tilde{\lambda}^+(\kappa_{m}) \\
	& \lambda^-(\kappa_{m})\lambda^+(\kappa_{m})=\tilde{\lambda}^-(\kappa_{m})\tilde{\lambda}^+(\kappa_{m}) \end{aligned}\right.
	\end{equation*}
\end{proposition}

\begin{proof} We start with the following Lemma:

\begin{lemma}
	\label{premeg}
There is $\mathcal{L}_1\subset\mathbb{N}$ such that $\displaystyle \sum_{m\in\mathcal{L}_1}\frac{1}{m}=+\infty$ satisfying :
\begin{equation*}
	\bigg(\lambda^-(\kappa_{m})=\tilde{\lambda}^-(\kappa_m),\:\:\forall m\in\mathcal{L}_1\bigg)\quad {\rm{or}}\quad \bigg(\lambda^-(\kappa_{m})=\tilde{\lambda}^+(\kappa_m),\:\:\forall m\in\mathcal{L}_1\bigg)
\end{equation*}
\end{lemma}

\begin{proof}
Since $\Sigma\big(\Lambda_{g}(\lambda)\big)=\Sigma\big(\Lambda_{\tilde{g}}(\lambda)\big)$, we have the inclusion :

\begin{equation*}
	\{\lambda^-(\kappa_m),\:m\in\mathbb{N}\}\subset \{{\tilde{\lambda}}^-(\kappa_m),\:m\in\mathbb{N}\}\cup\{\tilde{\lambda}^+(\kappa_m),\:m\in\mathbb{N}\}.
\end{equation*}

\noindent Thus, there exists  a sequence of integers $(a_m)$ such that $\displaystyle \sum_{m\in\mathbb{N}}\frac{1}{a_m}=+\infty$ and :

\begin{center}
$\displaystyle \{\lambda^-(\kappa_{a_m}),\:m\in\mathbb{N}\}\subset \{{\tilde{\lambda}}^-(\kappa_m),\:m\in\mathbb{N}\}\quad$ or $\quad\{\lambda^-(\kappa_{a_m}),\:m\in\mathbb{N}\}\subset \{{\tilde{\lambda}}^+(\kappa_m),\:m\in\mathbb{N}\}$.
\end{center}

\medskip

\noindent 
For instance, let us study  the first case (since the second case is similar). We can find another sequence of integers $(\tilde{a}_m)$ such that :
\begin{equation*}
	\lambda^-({\kappa_{a_m}})=\tilde{\lambda}^-({\kappa_{\tilde{a}_m}}),\quad\forall m\in\mathbb{N}.
\end{equation*}
Thanks to Lemma \ref{vp}, one obtains :

\begin{equation*}
	a_m-\tilde{a}_m=\frac{\tilde{h}'(1)}{4}-\frac{h'(1)}{4}+O\bigg(\frac{1}{a_m}\bigg).
\end{equation*}

\noindent Therefore, the sequence of integers $(a_m-\tilde{a}_m)$ converges to the integer $\displaystyle \frac{\tilde{h}'(1)}{4}-\frac{h'(1)}{4}=\frac{n-2}{4}\frac{f'(1)}{f(1)}-\frac{n-2}{4}\frac{\tilde{f}'(1)}{\tilde{f}(1)}.$

\noindent 
Recalling that $\displaystyle f,\tilde{f}\in \mathcal{C}_b=\bigg\{f\in C^\infty([0,1]),\: \bigg|\frac{f'(i)}{f(i)}\bigg|\le \frac{1}{n-2},\:i\in\{0,1\}\bigg\}$, we get:
\begin{equation*}
	\bigg|\frac{n-2}{4}\frac{f'(1)}{f(1)}-\frac{n-2}{4}\frac{\tilde{f}'(1)}{\tilde{f}(1)}\bigg|\le\frac{1}{2}.
\end{equation*}
As this quantity must be an integer, we have proved that, for $m\ge m_0$ :
\begin{equation*}
	a_m=\tilde{a}_m.
\end{equation*}

\noindent We conclude the proof of the Lemma setting $\mathcal{L}_1=\{a_m,\: m\ge m_0\}$.

\end{proof}

$\quad$

\noindent Now, we can finish the proof of the Proposition. For instance, assume that:
\begin{equation*}
\lambda^-(\kappa_{m})=\tilde{\lambda}^-(\kappa_m),\:\:\forall m\in\mathcal{L}_1.
\end{equation*}
Repeating the previous argument with $\lambda^+(\kappa_m)$, $m \in \mathcal{L}_1$, we have the inclusion:

\begin{equation*}
\{\lambda^+(\kappa_m),\:m\in\mathcal{L}_1\}\subset \{{\tilde{\lambda}}^-(\kappa_m),\:m\in\mathbb{N}\}\cup\{\tilde{\lambda}^+(\kappa_m),\:m\in\mathbb{N}\}
\end{equation*}

\noindent There is a sequence of integers, denoted $(b_m)\in\mathcal{L}_1$, such that $\displaystyle \sum_{m\in\mathbb{N}}\frac{1}{b_m}=+\infty$ and :

\begin{center}
	$\displaystyle \{\lambda^+(\kappa_{b_m}),\:m\in\mathbb{N}\}\subset \{{\tilde{\lambda}}^-(\kappa_m),\:m\in\mathbb{N}\}\quad$ or $\quad\{\lambda^+(\kappa_{b_m}),\:m\in\mathbb{N}\}\subset \{{\tilde{\lambda}}^+(\kappa_m),\:m\in\mathbb{N}\}$.
\end{center}

\begin{enumerate}
\item[$\bullet$] In the second case, there is a sequence $(\tilde{b}_m)$ such that :
\begin{equation*}
	\lambda^+(\kappa_{b_m})=\tilde{\lambda}^+(\kappa_{\tilde{b}_m})
\end{equation*}
which implies as previously that $b_m=\tilde{b}_m$. Then, setting $\mathcal{L}=\{b_m,\: m\in\mathbb{N}\}$ we have for all $m\in\mathcal{L}$ :
\begin{equation*}
\left\{
\begin{aligned}
&\lambda^-(\kappa_m)=\tilde{\lambda}^-({\kappa_m})	 \\
&\lambda^+(\kappa_m)=\tilde{\lambda}^+({\kappa_m})
\end{aligned}\right.
\end{equation*}

\item[$\bullet$] In the first case, we still could show that :
\begin{equation*}
\lambda^+(\kappa_{b_m})=\tilde{\lambda}^-(\kappa_{b_m})\:\:\big(=\lambda^-(\kappa_{b_m})\big)
\end{equation*}
\noindent But this is not possible because $\Lambda^{b_m}_g(\lambda)$ would then be an homothety, which contradicts the matrix expression of $\Lambda^{b_m}_g(\lambda)$.
\end{enumerate}
\noindent Hence the proof of Proposition \ref{eq_trans_1} is complete.

\end{proof}

\noindent We recall that $\mathcal{P}$ is the set of roots of $\Delta(z)$. For every $z\in\mathbb{C}\backslash\mathcal{P}$, let us set 
\begin{eqnarray}
\label{BC_1}
B(z)&=&-\frac{D(z)+E(z)}{\Delta(z)}+C_0-C_1 \\
\label{BC_2}
C(z)&=&\bigg(-\frac{D(z)}{\Delta(z)}+C_0\bigg)\bigg(-\frac{E(z)}{\Delta(z)}
-C_1\bigg)-\frac{1}{\sqrt{f(0)f(1)}\Delta(z)^2},
\end{eqnarray}
where 
\begin{center}
	$\displaystyle C_0=\frac{\ln(h)'(0)}{4\sqrt{f(0)}},\quad$ $\quad\displaystyle C_1=\frac{\ln(h)'(1)}{4\sqrt{f(1)}}$.
\end{center}

\noindent For every $m$ in $\mathbb{N}$, we have (cf proof of Lemma \ref{vp}):

\begin{center}
	$B(\kappa_m)=\lambda^-(\kappa_m)+\lambda^+(\kappa_m)\quad$ and $\quad C(\kappa_m)=\lambda^-(\kappa_m)\lambda^+(\kappa_m)$.
\end{center}

\medskip

\noindent We introduce the functions $g_1$ and $g_2$ on $\Pi^+$ as follows:
\begin{equation*}
	\begin{aligned}
	g_1(z)&= \Delta(z^2)\tilde{\Delta}(z^2)\big[B(z^2)-\tilde{B}(z^2)\big]\\
	g_2(z)&= \Delta(z^2)^2\tilde{\Delta}(z^2)^2\big[C(z^2)-\tilde{C}(z^2)\big].
	\end{aligned}
\end{equation*}
We claim that $g_1$ and $g_2$ are identically zero. Indeed :

\begin{enumerate}[label=\alph*), align=left, leftmargin=*, noitemsep]
	\item[$\bullet$] $g_1,g_2$ are holomorphic on $\Pi^+$.
	
	\medskip
	
	\item[$\bullet$] $g_1,g_2\in\mathcal{N}(\Pi^+)$ thanks to the estimates of Proposition \ref{estim}.
	
	\medskip
	
	\item[$\bullet$] Thanks to Proposition \ref{eq_trans_1}, we have $g_1(\sqrt{\kappa_m})=g_2(\sqrt{\kappa_m})=0$ for every $m\in\mathcal{L}$. As $\sqrt{\kappa_m}\sim m$, one has $\displaystyle \sum_{m\in\mathcal{L}}\frac{1}{\sqrt{\kappa_m}}=+\infty$. Thus, we can conclude, by Nevanlinna's theorem :
	\begin{equation*}
	g_1\equiv g_2\equiv 0 \:\: {\rm{on}}\:\:\Pi^+.
	\end{equation*}
\end{enumerate}

\noindent In particular, for every $m$ in $\mathbb{N}$ :

\begin{equation*}
\left\{
\begin{aligned}
&\lambda^-(\kappa_m)+\lambda^+(\kappa_m)=\tilde{\lambda}^-({\kappa_m})+\tilde{\lambda}^+({\kappa_m})	 \\
&\lambda^-(\kappa_m)\lambda^+(\kappa_m)=\tilde{\lambda}^-({\kappa_m})\tilde{\lambda}^+({\kappa_m})
\end{aligned}\right.
\end{equation*}

\medskip

\noindent and, by replacing $\kappa_m$ by $\mu_m$ we have for every $m$ in $\mathbb{N}$ :

\medskip

\begin{center}
	Tr$(\Lambda^m_g)=\: $Tr$(\Lambda^m_{\tilde{g}})$ and $\det(\Lambda^m_g)=\det(\Lambda^m_{\tilde{g}})$.
\end{center}

	\subsection{The case ${\bf f(0)\ne f(1)}$}

\medskip

\noindent Without loss of generality, assume that $f(0)<f(1)$. Thanks to the asymptotics:

\begin{equation*}
\sqrt{\kappa_m}=m+\frac{n-2}{2}+O\bigg(\frac{1}{m}\bigg)
\end{equation*} and Lemma \ref{vp}, we have :
	\begin{equation*}
\left\{
\begin{aligned}
&\lambda^-(\kappa_m)=\frac{m}{\sqrt{f(1)}}+\frac{n-2}{2\sqrt{f(1)}}-\frac{\ln(h)'(1)}{4\sqrt{f(1)}}+O\bigg(\frac{1}{m}\bigg)	 \\
&\lambda^+(\kappa_m)=\frac{m}{\sqrt{f(0)}}+\frac{n-2}{2\sqrt{f(0)}}+\frac{\ln(h)'(0)}{4\sqrt{f(0)}}+O\bigg(\frac{1}{m}\bigg).
\end{aligned}\right.
\end{equation*}

$\quad$

\noindent 
Since $f(0)\not= f(1)$, the proof is a little more delicate. As previously, we have to find a subset $\mathcal{L}\subset\mathbb{N}$ such that $\displaystyle \sum_{m\in\mathcal{L}}\frac{1}{m}=+\infty$ and such that $\lambda^-(\kappa_m)=\tilde{\lambda}^-(\kappa_m)$ or $\lambda^+(\kappa_m)=\tilde{\lambda}^+(\kappa_m)$ for all $m  \in \mathcal{L}$. We have the following Proposition:

\begin{proposition}
	\label{inclusion_1}
	There is $\mathcal{L}\subset \mathbb{N}$ satisfying $\displaystyle \sum_{m\in\mathcal{L}}\frac{1}{m}=+\infty$ and such that :
\begin{equation*}
\begin{aligned}
	&\{\lambda^-(\kappa_m),\:\: m\in\mathcal{L}\}\subset \{\tilde{\lambda}^-(\kappa_m),\:\: m\in\mathbb{N}\}\\
	\end{aligned}
	\end{equation*}
\end{proposition}

\begin{proof}
	
If the subset of the elements of $\{\lambda^-(\kappa_m),\:\: m\in\mathbb{N}\}$ belonging to $\{\tilde{\lambda}^+(\kappa_m),\:\: m\in\mathbb{N}\}$ is finite, the proposition is true. If not, there exist two strictly increasing sequences  $\varphi,\psi :\mathbb{N}\to\mathbb{N}$ such that:
\begin{equation}
	\label{egu}
	\lambda^-(\kappa_{\psi(m)})=\tilde{\lambda}^+(\kappa_{\varphi(m)}).
\end{equation}
	
	\medskip
	
\noindent	
We emphasize that the functions $\varphi$ and $\psi$ are built so that an integer $m\in\mathbb{N}$ which is not in the image of $\psi$ (respectively in that of $\varphi$) satisfies $\lambda^-(\kappa_m)=\tilde{\lambda}^-(\kappa_n)$ for some $n\in\mathbb{N}$ (respectively $\lambda^+(\kappa_n)=\tilde{\lambda}^+(\kappa_m)$ for some $n\in\mathbb{N}$).

		$\quad$

\noindent 
Replacing  $\lambda^+(\kappa_{\varphi(m)})$ and $\tilde{\lambda}^-(\kappa_{\psi(m)})$ by their asymptotics in the equality (\ref{egu}), we have :
\begin{equation*}
\label{asympt}
\frac{\varphi(m)}{\sqrt{f(0)}}+\frac{\ln(h)'(0)}{4\sqrt{f(0)}}+\frac{n-2}{2\sqrt{f(0)}}+O\bigg(\frac{1}{\varphi(m)}\bigg)=\frac{\psi(m)}{\sqrt{f(1)}}-\frac{\ln(h)'(1)}{4\sqrt{f(1)}}+\frac{n-2}{2\sqrt{f(1)}}+O\bigg(\frac{1}{\psi(m)}\bigg).
\end{equation*}
Setting $\displaystyle C=-\frac{\ln(h)'(1)}{4\sqrt{f(1)}}-\frac{\ln(h)'(0)}{4\sqrt{f(0)}}+\frac{n-2}{2\sqrt{f(1)}}-\frac{n-2}{2\sqrt{f(0)}}$ and noticing that the previous equality implies that $\displaystyle O\bigg(\frac{1}{\psi(m)}\bigg)=O\bigg(\frac{1}{\varphi(m)}\bigg)$, one can write (\ref{egu}) as follows:
\begin{equation}
	\label{asympt}
	\frac{\varphi(m)}{\sqrt{f(0)}}=\frac{\psi(m)}{\sqrt{f(1)}}+C+O\bigg(\frac{1}{\varphi(m)}\bigg).
\end{equation}

\vspace{0.5cm}
\noindent
Now, we have the following Lemma:

\begin{lemma}
		\label{image}
		There exists $m_0\in\mathbb{N}$ such that for all $m\ge m_0, \ \psi(m+1)\ge\psi(m)+ 2$.
\end{lemma}

\begin{proof}
		Set $\displaystyle A=\frac{\sqrt{f(1)}}{\sqrt{f(0)}}>1$ and $C'=-\sqrt{f(1)}C$. It follows from (\ref{asympt}) that:
		\begin{equation*}
		\psi(m)=A\varphi(m)+C'+O\bigg(\frac{1}{\varphi(m)}\bigg).
		\end{equation*}
		Assume $\psi(m+1)=\psi(m)+1$. Then :
		\begin{equation*}
		\begin{aligned}
		\psi(m)+1=\psi(m+1)&=A\varphi(m+1)+C'+O\bigg(\frac{1}{\varphi(m)}\bigg)\\
		&\ge A(\varphi(m)+1)+C'+O\bigg(\frac{1}{\varphi(m)}\bigg)\\
		&=A\varphi(m)+C'+A+O\bigg(\frac{1}{\varphi(m)}\bigg)\\
		&=\psi(m)+A+O\bigg(\frac{1}{\varphi(m)}\bigg).
		\end{aligned}
		\end{equation*}
Thus, we get:
		\begin{equation*}
		1\ge A+O\bigg(\frac{1}{\varphi(m)}\bigg)
		\end{equation*}
which is clearly false for $m$  large enough.	
\end{proof}

\medskip

\noindent 
Consequently, the range of $\psi$  doesn't contain two consecutive integers. We deduce from this the following Lemma :

\medskip

\begin{lemma}
	\label{lemman}
Denote $a_1,a_2,...$ the sequence of all integers that are not in the range of $\psi$. There exists $C>0$ such that this sequence satisfies :
\begin{equation}
\label{inan}
	a_m\le 2m+C.
\end{equation}
\end{lemma}

\begin{proof}

\noindent We set $p_0=\psi(m_0)$ and $C:= a_{p_0} -2p_0$. Clearly, (\ref{inan}) is true for $m=p_0$. Now, assume that (\ref{inan}) is satisfied for a fixed $m\ge p_0$. Then :
\begin{enumerate}
	\item[$\bullet$] either $a_m+1$ is not in the range of $\psi$ and so : $a_{m+1}=a_m+1$,
	\item[$\bullet$] or $a_m+1$ is in the range of $\psi$. Then $a_m+2$ is not in the image of $\psi$ since it does not contain consecutive integers and consequently $a_{m+1}=a_{m}+2$.
\end{enumerate}
In both cases, we get $\displaystyle a_{m+1}\le 2(m+1)+C$. Then, the proof of the Lemma follows by a standard induction argument.
\end{proof}

\noindent
In particular, Lemma \ref{lemman} shows that :
\begin{equation*}
\sum_{n\in\mathbb{N}}\frac{1}{a_n}=+\infty.
\end{equation*}

\medskip

\noindent 
Thus, setting $\mathcal{L}=\{a_n,\:\: n\in\mathbb{N}\}$, we have :

\medskip

\begin{equation*}
\begin{aligned}
&\{\lambda^-(\kappa_m),\:\: m\in\mathcal{L}\}\subset \{\tilde{\lambda}^-(\kappa_m),\:\: m\in\mathbb{N}\}\\
\end{aligned}
\end{equation*}
which concludes  the proof of the Proposition.
\end{proof}

\vspace{0.5cm}
\noindent
As previously, we have to treat differently the case of the dimension $n=2$ and $n\geq 3$.


\subsubsection*{The two dimensional case:}

\medskip

	\begin{lemma}
		\label{lemma_eq_2}
	For all $m$ $\in\mathcal{L}$, one has:
	\begin{equation*}
	\lambda^-(\kappa_m)=\tilde{\lambda}^-(\kappa_m).
	\end{equation*}
	\end{lemma}

\begin{proof}

\noindent By construction, for each element $m$ of $\mathcal{L}$, there exists $\ell(m)\in\mathbb{N}$ such that :
\begin{equation*}
\displaystyle \lambda^-(\kappa_m)=\tilde{\lambda}^-(\kappa_{\ell(m)}).
\end{equation*}
Using Lemma \ref{vp}, we get:
\begin{equation*}
m=\ell(m)+O\bigg(\frac{1}{m}\bigg)+O\bigg(\frac{1}{\ell(m)}\bigg),
\end{equation*}
which implies $m=\ell(m)$ for $m$ large enough.
\end{proof}

$\quad$

\noindent Let us consider again the functions $B$ and $C$ defined in (\ref{BC_1}) and (\ref{BC_2}). Setting
\begin{equation*}
R(z)=B(z)^2-4C(z),
\end{equation*} we get:
\begin{equation*}
\lambda^{\pm}(\kappa_m)=\frac{1}{2}\bigg(B(\kappa_m)\pm\sqrt{R(\kappa_m)}\bigg).
\end{equation*}

\noindent We can define on $\Pi^+$ the function $\lambda^-(z)$ by $\displaystyle \lambda^-(z)= \frac{1}{2}\bigg(B(z)-\sqrt{R(z)}\bigg)$.

\begin{lemma}
	\label{alt2}
	For every $t\in\mathbb{R}_+$ large enough, we have the alternative :
	
	\medskip
	
	\begin{center}
		$\lambda^-(t)=\tilde{\lambda}^-(t)\quad$ or $\quad\lambda^+(t)=\tilde{\lambda}^+(t)$.
	\end{center}
\end{lemma}

\begin{proof}
From Lemma \ref{lemma_eq_2}, there is $\mathcal{L}\subset\mathbb{N}$ satisfying $\displaystyle \sum_{m\in\mathcal{L}}\frac{1}{m}=+\infty$ and such that

\begin{equation*}
	\forall m\in\mathcal{L},\qquad B(\kappa_m)-\sqrt{R(\kappa_m)}=\tilde{B}(\kappa_m)-\sqrt{\tilde{R}(\kappa_m)}.
	\end{equation*}
One has
	\begin{equation*}
	\bigg(B(\kappa_{m})-\tilde{B}(\kappa_{m})\bigg)^2=\bigg(\sqrt{R(\kappa_{m})}-\sqrt{\tilde{R}(\kappa_{m})}\bigg)^2,
	\end{equation*}
and then
\begin{equation*}
	\underbrace{\bigg[\bigg(B(\kappa_{m})-\tilde{B}(\kappa_{m})\bigg)^2-R(\kappa_{m})-\tilde{R}(\kappa_m)\bigg]^2-4R(\kappa_{m})\tilde{R}(\kappa_{m})}_{:=g_1(\kappa_m)}=0.
\end{equation*}
For $z \in \Pi^+$, let us define:
\begin{equation*}
	g_2(z)=\Delta(z^2)^4\tilde{\Delta}(z^2)^4g_1(z^2).
\end{equation*}
	
	\medskip
	
	\noindent Then $g_2$ is identically zero. Indeed :
	
	\medskip
	
	\begin{enumerate}[label=\alph*), align=left, leftmargin=*, noitemsep]
		\item[$\bullet$] $g_2$ is holomorphic on $\Pi^+$.
		
		\medskip
		
		\item[$\bullet$] $g_2\in\mathcal{N}(\Pi^+)$ thanks to the estimates of Proposition \ref{estim}.
		
		\medskip
		
		\item[$\bullet$] We have $g_2(\sqrt{\kappa_m})=0$ for every $m\in\mathcal{L}$, which enable us to conclude, by Nevanlinna's theorem, that
		\begin{equation*}
		g_2\equiv 0 \:\: {\rm{on}}\:\:\Pi^+.
		\end{equation*}
		
	\end{enumerate}
We have obtained:
\begin{equation*}
	\forall z\in\Pi^+,\qquad\bigg[\bigg(B(z^2)-\tilde{B}(z^2)\bigg)^2-R(z^2)-\tilde{R}(z^2)\bigg]^2-4R(z^2)\tilde{R}(z^2)=0.
\end{equation*}
It easily follows that for $z\in\Pi^+$, we have four alternatives:
	\begin{equation*}
	\lambda^-(z^2) = \tilde{\lambda}^-(z^2),\quad\lambda^+(z^2) = \tilde{\lambda}^-(z^2),\quad \lambda^-(z^2) =\tilde{\lambda}^+(z^2)\quad {\rm{or}}\quad \lambda^+(z^2) = \tilde{\lambda}^+(z^2).
	\end{equation*}

\noindent We can note that the asymptotics of Lemma \ref{vp}, involving $\lambda^\pm(\kappa_m)$, can be extended on $\mathbb{R}_+$ for $\lambda^\pm(t)$ (the proof remains exactly the same). Now, if a real sequence of positive numbers $(t_m)$ satisfies, for instance, $\displaystyle\lambda^+(t_m^2) = \tilde{\lambda}^-(t_m^2)$ with $t_m\to +\infty$, then, using Lemma \ref{vp} with $t_m^2$ instead of $\kappa_m$ we get $f(0)=f(1)$, and that contradicts our hypothesis. We use the same argument to exclude the third case when $z$ is a positive real number large enough. Finally, we have for $t$ positive large enough :
	\begin{equation*}
	\lambda^-(t) = \tilde{\lambda}^-(t)\quad {\rm{or}}\quad \lambda^+(t) = \tilde{\lambda}^+(t).
	\end{equation*}
\end{proof}
\noindent \textit{Case 1} : Assume now there exists $x\in\mathbb{R}_+$ as large as we want such that $\displaystyle \lambda^-(x) \ne \tilde{\lambda}^-(x)$. By a standard continuity argument, we can find an interval $I$ centered in $x$ such that :

\begin{center}
$\displaystyle\forall t\in I,\: \lambda^-(t) \ne \tilde{\lambda}^-(t)$,
\end{center} 
which implies necessarily that $\displaystyle \forall t\in I,\: \lambda^+(t) = \tilde{\lambda}^+(t)$. Moreover, there is $L>0$ such that the real function
\begin{equation*}
t\mapsto\lambda^+(t)-\tilde{\lambda}^+(t)
\end{equation*}  is analytic on the interval $[L,+\infty[$. Thus, $\forall t\ge L$, one has $\lambda^+(t) = \tilde{\lambda}^+(t)$ by the analytic continuation principle. One deduces there exists $m_0\in\mathbb{N}$ such that, for $m\in\mathcal{L}$ and $m\ge m_0$:
$$
\lambda^+(\kappa_m) = \tilde{\lambda}^+(\kappa_m) \ \quad\mathrm{and}\quad\  \lambda^-(\kappa_m) = \tilde{\lambda}^-(\kappa_m).
$$
Without loss of generality, assume that $\min\mathcal{L}$ is greater than $m_0$. As a by product, one gets:
	\begin{equation*}
\forall m\in\mathcal{L},\:\:\left\{
\begin{aligned}
& \lambda^-(\kappa_{m})+\lambda^+(\kappa_{m})=\tilde{\lambda}^-(\kappa_{m})+\tilde{\lambda}^+(\kappa_{m}) \\
& \lambda^-(\kappa_{m})\lambda^+(\kappa_{m})=\tilde{\lambda}^-(\kappa_{m})\tilde{\lambda}^+(\kappa_{m}) \end{aligned}\right.
\end{equation*}
\noindent Hence, we have proved the same result as the one of Proposition \ref{eq_trans_1}. We can deduce similarly that both previous equalities are in fact true for every $m\in\mathbb{N}$, i.e 
\begin{center}
	$\forall m\in\mathbb{N},\quad$ Tr$(\Lambda^m_g)=\: $Tr$(\Lambda^m_{\tilde{g}})\quad$ and $\quad\det(\Lambda^m_g)=\det(\Lambda^m_{\tilde{g}})$.
\end{center}
\noindent We are thus brought back to the third section.

$\quad$

\noindent \textit{Case 2} : For all $t\in\mathbb{R}$, $t$ large enough, we have $\lambda^-(t)=\tilde{\lambda}^-(t)$. In particular, for all $m\in\mathbb{N}$ large enough, $\lambda^-(\kappa_m)=\tilde{\lambda}^-(\kappa_m)$. Hence, for $m$ large enough, one also has $\lambda^+(\kappa_m)=\tilde{\lambda}^+(\kappa_m)$. Indeed, let $m\in\mathbb{N}$. If $\lambda^+(\kappa_m)\in \big(\tilde{\lambda}^+(\kappa_\ell)\big)$, there is $\ell(m)\in\mathbb{N}$ such that $\lambda^+(\kappa_m)=\tilde{\lambda}^+(\kappa_{\ell(m)})$. One can prove, as in the proof of Lemma \ref{lemma_eq_2} that $\ell(m)=m$. The same holds if we assume that $\tilde{\lambda}^+(\kappa_m)\in \big(\lambda^+(\kappa_\ell)\big)$. The only other option is that there are integers $p$ and $\ell$ such that 
\begin{equation*}
	\left\{\begin{aligned}
	&\lambda^+(\kappa_m) = \tilde{\lambda}^-(\kappa_p)\\
	&\tilde{\lambda}^+(\kappa_m)= \lambda^-(\kappa_\ell)
	\end{aligned}\right.
\end{equation*}
\noindent Then $\lambda^+(\kappa_m)-\tilde{\lambda}^+(\kappa_m)=\tilde{\lambda}^-(\kappa_p)-\lambda^-(\kappa_\ell)$, which implies
\begin{equation*}
	O\bigg(\frac{1}{m}\bigg)=\frac{p-\ell}{\sqrt{f(1)}}+O\bigg(\frac{1}{p}\bigg)+O\bigg(\frac{1}{\ell}\bigg).
\end{equation*}
\noindent Hence $p=\ell$ for $m,p,\ell$ large enough and so $\displaystyle \lambda^+(\kappa_m)=\tilde{\lambda}^+(\kappa_m)$. There is then $m_0\in\mathbb{N}$ such that
\begin{center}
	$\forall m\ge m_0,\quad$ Tr$(\Lambda^m_g)=\: $Tr$(\Lambda^m_{\tilde{g}})\quad$ and $\quad\det(\Lambda^m_g)=\det(\Lambda^m_{\tilde{g}})$.
\end{center}
\noindent We are brought back again to the third section (see Remark \ref{Remarque_importante_1} p. 13).

\subsubsection*{The case of the  dimension ${\bf n\ge 3}$.}

\medskip

\noindent The following arguments are more or less immediate adaptations of the two dimensional case. We just refer to it for details.

\medskip

\begin{lemma}
	There exists $\mathcal{L}\subset\mathbb{N}$ such that  $\displaystyle\sum_{m\in\mathcal{L}}\frac{1}{m}=+\infty$ and :
	\begin{equation*}
	\tilde{\lambda}^-(\kappa_{m})=\lambda^-(\kappa_m), \quad\forall m\in\mathcal{L}.
	\end{equation*}
\end{lemma}

\begin{proof}
By definition of the sequence $(a_n)$ defined in Lemma \ref{lemman}, there is another sequence of integers $(\tilde{a}_n)$ such that :
	\begin{equation*}
	\lambda^-(\kappa_{a_n})=\tilde{\lambda}^-(\kappa_{\tilde{a}_n})
	\end{equation*}
As in the proof of Lemma \ref{premeg}, we show that  $a_m=\tilde{a}_m$ and we set $\mathcal{L}=\{a_m,\: m\in\mathcal{L}\}$.
\end{proof}

\vspace{0.2cm}\noindent
We have the following Lemma (the proof is identical to the two dimensional case).
\begin{lemma}
For all $t\in\mathbb{R}_+$ large enough, we have the alternative :
\begin{equation*}
\lambda^-(t) = \tilde{\lambda}^-(t),\quad\quad \lambda^+(t) = \tilde{\lambda}^+(t).
\end{equation*}
\end{lemma}

\vspace{0.5cm}\noindent
We deduce then, for all $m\in\mathcal{L}$ :
\begin{center}
$\lambda^-(\kappa_{m}) = \tilde{\lambda}^-(\kappa_{m})\quad$ and $\quad \lambda^+(\kappa_{m}) = \tilde{\lambda}^+(\kappa_m)$.
\end{center}

\noindent 
It follows that :

\begin{center}
	$\forall m\in\mathbb{N},\:\:$ Tr$(\Lambda^m_g)=\: $Tr$(\Lambda^m_{\tilde{g}})$ and $\det(\Lambda^m_g)=\det(\Lambda^m_{\tilde{g}})$.
\end{center}

$\quad$

\noindent {\bf Aknowledgements:} The author would like to thank Thierry Daud\'e and Fran\c{c}ois Nicoleau for their encouragements and helpful discussions. The author also thanks the referees for valuable remarks and comments that improved the first version of this paper.

\bibliographystyle{acm}
\bibliography{bibliographie}

\begin{thebibliography}{10}

\bibitem{agranovich2006mixed}
{\sc Agranovich, M.}
\newblock On a mixed {P}oincar{\'e}-{S}teklov type spectral problem in a
  {L}ipschitz domain.
\newblock {\em Russian Journal of Mathematical Physics 13}, 3 (2006), 239--244.

\bibitem{colbois2019steklov}
{\sc Colbois, B., Girouard, A., and Hassannezhad, A.}
\newblock The {S}teklov and {L}aplacian spectra of {R}iemannian manifolds with
  boundary.
\newblock {\em Journal of Functional Analysis\/} (2019), 108409.

\bibitem{daude2015non}
{\sc Daud{\'e}, T., {K}amran, N., and {N}icoleau, F.}
\newblock Non uniqueness results in the anisotropic {C}alder{\'o}n problem with
  {D}irichlet and {N}eumann data measured on disjoint sets.
\newblock In {\em {A}nnales de l’{I}nstitut {F}ourier\/} (2019), vol.~49.

\bibitem{daude2019hidden}
{\sc Daud{\'e}, T., Kamran, N., and Nicoleau, F.}
\newblock On the hidden mechanism behind non-uniqueness for the anisotropic
  {C}alder{\'o}n problem with data on disjoint sets.
\newblock In {\em Annales Henri Poincar{\'e}\/} (2019), vol.~20, Springer,
  pp.~859--887.

\bibitem{daude2018stability}
{\sc Daud{\'e}, T., Kamran, N., and Nicoleau, F.}
\newblock Stability in the inverse {S}teklov problem on warped product
  riemannian manifolds.
\newblock {\em The Journal of Geometric Analysis\/} (2019).

\bibitem{ferreira2009limiting}
{\sc Ferreira, D. D.~S., Kenig, C.~E., Salo, M., and Uhlmann, G.}
\newblock Limiting {C}arleman weights and anisotropic inverse problems.
\newblock {\em Inventiones mathematicae 178}, 1 (2009), 119--171.

\bibitem{girouard2019steklov}
{\sc Girouard, A., Lagac{\'e}, J., Polterovich, I., and Savo, A.}
\newblock The {S}teklov spectrum of cuboids.
\newblock {\em Mathematika 65}, 2 (2019), 272--310.

\bibitem{girouard2017spectral}
{\sc Girouard, A., and Polterovich, I.}
\newblock Spectral geometry of the {S}teklov problem (survey article).
\newblock {\em Journal of Spectral Theory 7}, 2 (2017), 321--360.

\bibitem{jollivet2014inverse}
{\sc Jollivet, A., and Sharafutdinov, V.}
\newblock On an inverse problem for the {S}teklov spectrum of a {R}iemannian
  surface.
\newblock {\em Contemp. Math 615\/} (2014), 165--191.

\bibitem{kohn1984identification}
{\sc Kohn, R.~V., and Vogelius, M.}
\newblock Identification of an unknown conductivity by means of measurements at
  the boundary.
\newblock In {\em SIAM-AMS Proceedings\/} (1984), American Mathematical Soc.

\bibitem{lionheart1997conformal}
{\sc Lionheart, W.}
\newblock Conformal uniqueness results in anisotropic electrical impedance
  imaging.
\newblock {\em Inverse Problems 13}, 1 (1997), 125.

\bibitem{parzanchevski2013g}
{\sc Parzanchevski, O.}
\newblock On $ g $-sets and isospectrality.
\newblock In {\em Annales de l'Institut Fourier\/} (2013), vol.~63,
  pp.~2307--2329.

\bibitem{poschel1987inverse}
{\sc {P}oschel, J., and {T}rubowitz, E.}
\newblock {\em Inverse spectral theory}, vol.~130.
\newblock Academic Press, 1987.

\bibitem{provenzano2019weyl}
{\sc Provenzano, L., and Stubbe, J.}
\newblock Weyl-type bounds for {S}teklov eigenvalues.
\newblock {\em Journal Of Spectral Theory 9}, ARTICLE (2019), 349--377.

\bibitem{ramm1999inverse}
{\sc Ramm, A.}
\newblock An inverse scattering problem with part of the fixed-energy phase
  shifts.
\newblock {\em Communications in mathematical physics 207}, 1 (1999), 231--247.

\bibitem{safarov1997asymptotic}
{\sc Safarov, J., and Vassilev, D.}
\newblock {\em The asymptotic distribution of eigenvalues of partial
  differential operators}, vol.~155.
\newblock American Mathematical Soc., 1997.

\bibitem{salo2013calderon}
{\sc Salo, M.}
\newblock The {C}alder{\'o}n problem on {R}iemannian manifolds.
\newblock {\em Inverse problems and applications: inside out. II, Math. Sci.
  Res. Inst. Publ 60\/} (2013), 167--247.

\bibitem{shubin1987pseudodifferential}
{\sc Shubin, M.~A.}
\newblock {\em Pseudodifferential operators and spectral theory}, vol.~200.
\newblock Springer, 1987.

\bibitem{simon1999new}
{\sc Simon, B.}
\newblock A new approach to inverse spectral theory, {I.} fundamental
  formalism.
\newblock {\em Annals of Mathematics-Second Series 150}, 3 (1999), 1029--1058.

\bibitem{uhlmann2009electrical}
{\sc Uhlmann, G.}
\newblock Electrical impedance tomography and {C}alder{\'o}n's problem.
\newblock {\em Inverse problems 25}, 12 (2009), 123011.

\end{thebibliography}


  \end{document}